\documentclass[
	11pt,
]{amsart}%
\usepackage{tikz}
\usepackage{rotating}
\usetikzlibrary{shapes,arrows}
\usepackage[in]{fullpage}
\usepackage[latin1]{inputenc}
\usepackage{algorithm, algorithmic}
\usepackage{enumerate}
\usepackage{amsmath}\usepackage{amsthm}\usepackage{amssymb}\usepackage{mathrsfs}\usepackage{amscd}\usepackage{graphicx}\usepackage{subfigure}\usepackage{amsfonts}\usepackage{amsxtra}\usepackage{color}
\usepackage{multirow}
\usepackage{amscd}
\usepackage{fullpage}

\DeclareMathOperator{\rank}{rank}

\DeclareMathOperator{\range}{range}

\DeclareMathOperator{\trace}{trace}

\DeclareMathOperator{\diag}{diag}\DeclareMathOperator*{\spann}{span}

\theoremstyle{definition}\newtheorem{definition}{Definition}[section]\newtheorem{example}[definition]{Example}\newtheorem{remark}[definition]{Remark}
\theoremstyle{plain}\newtheorem{theorem}[definition]{Theorem}\newtheorem{lemma}[definition]{Lemma}\newtheorem{corollary}[definition]{Corollary}\newtheorem{proposition}[definition]{Proposition}

\newcommand{\R}{\mathbb{R}}\newcommand{\C}{\mathbb{C}}\newcommand{\N}{\mathbb{N}}

\newcommand{\K}{\mathbb{K}}

\begin{document}

\title{Preconditioning filter bank decompositions using structured normalized tight frames
}
\author{Martin Ehler}
\address[M.~Ehler]{University of Vienna,
Faculty of Mathematics, 
Oskar-Morgenstern-Platz 1
A-1090 Vienna
} 
\email{martin.ehler@univie.ac.at}

\begin{abstract}
We turn a given filter bank into a filtering scheme that provides perfect reconstruction, synthesis is the adjoint of the analysis part (so-called unitary filter banks), all filters have equal norm, and the essential features of the original filter bank are preserved. Unitary filter banks providing perfect reconstruction are induced by tight generalized frames, which enable signal decomposition using a set of linear operators. If, in addition, frame elements have equal norm, then the signal energy is spread through the various filter bank channels in some uniform fashion, which is often more suitable for further signal processing. We start with a given generalized frame whose elements allow for fast matrix vector multiplication, as for instance, convolution operators, and compute a normalized tight frame, for which signal analysis and synthesis still preserve those fast algorithmic schemes.

\end{abstract}




\maketitle 

\section{Introduction}
Increasingly detailed data are acquired in all sorts of measurements nowadays, so that fast algorithms are an important factor for successful signal processing. The concept of generalized frames has a long tradition in signal processing and many unitary filter bank schemes with the perfect reconstruction property are induced by tight generalized frames. Frames themselves are basis-like systems that span a vector space but allow for linear dependency. The inherent redundancy of frames can yield advantageous features unavailable within the basis concept \cite{Christensen:2003aa,Ehler:2010aa,Ehler:2010ac,Okoudjou:2010aa}. If the frame is \emph{tight} and its elements have \emph{unit norm}, then it resembles the concept of an orthonormal basis -- with the add-on of useful redundancy -- and  frame coefficients measure the signal energy in a uniform fashion. Generalized frames were introduced in \cite{Sun:2006fk} as a tool for signal decomposition using a set of linear operators. In \cite{Casazza:2008aa}, collections of orthogonal projectors were considered under the name fusion frames, fusion frame filter banks have been considered in \cite{Chebira:2011fk}, and the concept of tight $p$-fusion frames was developed in \cite{Bachoc:2010aa}. As convolution operators are linear, most filter banks can be thought of as pairs of generalized frames, one for analysis and the other for synthesis. Hence, in view of filter banks, it is not sufficient to deal with frames but we must inevitably consider their generalized counterpart. Tightness of a generalized frame means that the induced unitary filter bank provides perfect reconstruction. As with frames, we seek unit norm tight generalized frames because signal energy is then spread through the various channels in a more uniform fashion. The latter was used in \cite{Kutyniok:2009aa} to verify robustness of tight fusion frames against erasures, meaning that it is beneficial to have tight fusion frames with equal norm when dealing with distortions and loss of data. To keep the filter bank perspective, we shall focus on generalized frames consisting of convolution operators enabling fast algorithms.  

In the present paper we start with a generalized frame, whose elements allow for fast matrix vector multiplications (for instance, convolution operators) and construct a unit norm tight generalized frame that induces a filter bank scheme  preserving those fast algorithms. The latter is related to the so-called Paulsen problem for frames, where one is given a unit norm frame, and one asks for the closest tight frame with unit norm and for an algorithm to find it. This problem for frames has been partially solved in \cite{Bodmann:2010fk,Cahill:2011ys,Casazzaa:2010fk}. Note that if we are given a unit norm generalized frame, whose elements allow for fast matrix vector multiplications, then the closest tight generalized frame with unit norm may not provide such fast algorithmic schemes in general. Here, we aim to find a related tight unit norm generalized frame in such a way that signal analysis and synthesis can still benefit from the underlying fast matrix vector multiplications. 

We should point out that we used the term filter bank beyond sets of convolution operators similar to \cite{Chebira:2011fk}, where weighted orthogonal projectors are considered. Nonetheless, if the starting generalized frame consists of convolution operators, then the resulting scheme still represents convolution operators in each channel, but we require one additional linear operator for global pre- and postmultiplication. As this operator has a special structure being the inverse of a convolution frame operator, there are still fast computation schemes available \cite{Wiesmeyr:2013fk}. 

Our construction is inspired by pseudocovariance estimators of elliptical distributions in \cite{Tyler:1987fk}, see also \cite{Kent:1988kx,Tyler:1987uq}. We derive an iterative algorithm on positive definite matrices, for which we prove convergence, so that we obtain a positive definite matrix $\Gamma$ that enables us to construct the tight unit norm generalized frame.

\smallskip
The outline is as follows: In Section \ref{section:frames}, we introduce the concept of generalized frames and motivate the construction of unit norm tight generalized frames. In Section \ref{sec:I}, we present our iterative algorithm, for which we verify convergence, enabling us to construct tight generalized frames with unit norm that preserve fast analysis and synthesis due to their special structure. In Section \ref{sec:II} we provide few examples of random matrices whose samples satisfy the convergence assumptions needed. We also point out examples for convolution operators and further operators enabling fast matrix  vector multiplications. In Section \ref{sec:III}, we discuss the structure of our construction when the underlying generalized frame is a sample from an elliptical distribution. 

\section{Generalized frames}\label{section:frames}
We follow \cite{Sun:2006fk} and call a collection $\{T_j\}_{j=1}^n\subset \K^{d\times r}$ a \emph{generalized frame} (or a g-frame for short) if there are two constants $0<A\leq B<\infty$ such that 
\begin{equation}\label{eq:fusion frame formula}
A\|x\|^2 \leq \sum_{j=1}^N\|T^*_jx\|^2  \leq B\|x\|^2, \text{ for all $x\in\K^d$.}
\end{equation} 
If the constants can be chosen $0<A=B$, then $\{T_j\}_{j=1}^n$ is called a \emph{tight g-frame} and it is called a \emph{Parseval} g-frame if $0<A=B=1$. For $r=1$, we have $T_j^*x = \langle T_j,x\rangle$, so that we recover the concept of frames, cf.~\cite{Christensen:2003aa}. It turns out that a collection $\{T_j\}_{j=1}^n\subset \K^{d\times r}$ is a g-frame if and only if $\bigcup_{j=1}^n \range(T_j)$ spans $\K^d$. 

If $\{T_j\}_{j=1}^n$ is a g-frame, then the \emph{analysis operator} $\mathcal{F}: \K^d\rightarrow \K^{n\times r}$ is given by $x\mapsto (x^*T_j)_{j=1}^n$. Its adjoint is the \emph{synthesis operator}  $\mathcal{F}^*:\K^{n\times r}\rightarrow\K^d$, defined by $ (c_j)_{j=1}^n\mapsto \sum_{j=1}^n T_j c^*_j,
$  such that the generalized frame operator is 
\begin{equation*}
S=\mathcal{F}^*\mathcal{F}: \K^d \rightarrow \K^d,\qquad x\mapsto \sum_{j=1}^n T_jT^*_j x.
\end{equation*}
The collection $\{S^{-1}T_j\}_{j=1}^n$ is called the \emph{canonical dual g-frame} and yields the expansion
\begin{equation*}
x = \sum_{j=1}^n T_j (S^{-1}T_j)^* x=\sum_{j=1}^n S^{-1}T_j T_j^*x ,\qquad \text{for all } x\in\K^d,
\end{equation*}
which simply follows from $SS^{-1}=S^{-1}S=I$.

\begin{proposition}\label{th:Parse}
If $\{T_j\}_{j=1}^\infty\subset \K^{d\times r}$ is a g-frame with g-frame operator $S$, then $\{S^{-1/2}T_j\}_{j=1}^n$ is a Parseval g-frame and, for any other Parseval g-frame $\{R_j\}_{j=1}^\infty\subset\K^{d\times r}$, we have
\begin{equation}\label{EQ:INEQ}
\sum_{j=1}^n \|T_j-S^{-1/2}T_j\|^2_{HS}\leq \sum_{j=1}^n \|T_j-R_j\|^2_{HS}. 
\end{equation}
Equality holds if and only if $R_j=S^{-1/2}T_j$, for $j=1,\ldots,n$. 
\end{proposition}
Most parts of the proof of Proposition \ref{th:Parse} can follow the lines in \cite{G.-Kutyniok:2007fk}, where $r=1$ is considered, so we omit the proof.


\begin{remark}
We have supposed that the linear operators of a g-frame have all the same dimensions, which simplifies notation but is not necessary. The entire paper could also deal with sets of linear operators $\{T_j\}_{j=1}^n$, where $T_j\in\K^{d\times r_j}$, for $j=1,\ldots,n$. Then $T_jT^*_j\in\K^{d\times d}$ and this is all we need. 
\end{remark}

Tight frames are desirable because synthesis is simply the adjoint of the analysis part. For signal processing purposes, we are interested in tight g-frames that additionally have unit norm, because those more resemble orthonormal bases, and $\{\|T_j^*x\|\}_{j=1}^n$ then has more information about the signal energy in the direction of a particular frame element, see \cite{Benedetto:2003aa} for $r=1$.

Given some g-frame, say with unit norm elements, let us seek a tight g-frame with equal norm elements that is nearby. If we give up the equal norm requirement and $S$ is the frame operator of some g-frame $\{T_j\}_{j=1}^n$, then the collection $\{S^{-1/2}T_j\}_{j=1}^n$ is a Parseval frame closest to $\{T_j\}_{j=1}^n$, see Proposition \ref{th:Parse}. In general, however,  $\{S^{-1/2}T_j\}_{j=1}^n$ may not have equal norms. The search for the closest Parseval frame with equal norm elements has become known as the Paulsen problem. It is essentially the same problem if we restrict us to the sphere, i.e., given a unit norm g-frame, we aim for the closest unit norm tight g-frame. For $r=1$, this problem was partially solved in \cite{Bodmann:2010fk,Cahill:2011ys,Casazzaa:2010fk}.

Suppose now that we are given a g-frame $\{T_j\}_{j=1}^n$ that allows for fast matrix vector multiplications for each $T_j$ and $T^*_j$. The closest equal norm Parseval g-frame may not preserve such features. From a computational point of view, it would be preferable to find an equal norm Parseval g-frame that still allows for fast analysis and synthesis schemes, and this is indeed our topic in the subsequent sections.

\section{Constructing unit norm tight g-frames that preserve fast algorithms }\label{sec:I}
The $g$-frame operator $S$ of some g-frame $\{T_j\}_{j=1}^n$ and hence also $S^{-1}$ are positive definite and Proposition \ref{th:Parse} yields the tight $g$-frame $\{S^{-1/2}T_j\}_{j=1}^n$ that may not have unit norm, and $\{S^{-1/2}T_j/\|S^{-1/2}T_j\|_{HS}\}_{j=1}^n$ has unit norm but may not be tight. 
To construct a unit norm tight g-frame that preserves fast matrix vector multiplications, we get inspired by Proposition \ref{th:Parse} and aim to find a positive definite matrix $\Gamma$ such that 
\begin{equation}\label{eq:tight def apporach}
\big\{\frac{\Gamma^{1/2}T_j}{\|\Gamma^{1/2}T_j\|_{HS}}\big\}_{j=1}^n
\end{equation}
is a unit norm tight g-frame. As opposed to Proposition \ref{th:Parse}, we replace $S^{-1}$ with $\Gamma$ and normalize. The unit norm $g$-frame \eqref{eq:tight def apporach} is tight if and only if 
\begin{equation}\label{eq:tightness at the beginning}
I = \Gamma^{1/2} \frac{d}{n}\sum_{j=1}^n \frac{T_jT_j^*}{\|\Gamma^{1/2}T_j\|_{HS}^2} \Gamma^{1/2}.
\end{equation}
Signal analysis and synthesis requires pre- and postmultiplication with $\Gamma^{1/2}$ but inbetween we can use $n$ times the fast algorithms provided by $T_j^*$ and $T_j$, cf.~Fig.~\ref{fig:filtering}. Now, $T_j^*\Gamma^{1/2}/\|\Gamma^{1/2}T_j\|_{HS}$ has unit norm, so that the signal energy better relates to the magnitudes of $\{ \|T_j^*\Gamma^{1/2}x\|/\|\Gamma^{1/2}T_j\|_{HS}  \}_{j=1}^n$. Thus, the special structure \eqref{eq:tight def apporach} can be advantageous over other unit norm tight $g$-frames that can be closer to the original one. 

\begin{remark}
The filtering scheme in Fig.~\ref{fig:filtering} can preserve many properties of the original g-frame $\{T_j\}_{j=1}^n$, which can go beyond fast matrix vector multiplications, such as, being orthogonal projectors, sparse matrices, etc, as long as the application of $\Gamma^{1/2}$ is implemented separately and we do not use $\Gamma^{1/2}T_j$ directly. 
\end{remark}
\begin{remark}
We point out that the structure \eqref{eq:tight def apporach} is different from the approach in \cite{Kutyniok:2013fk}, where rescalings are seeked to derive tight frames. The authors in \cite{Frank:2002uq} discuss the setting when a linear operator exists that maps a frame into a unit norm tight frame. We are more general here, because we are joining both, we apply a linear operator and allow for rescaling. 
\end{remark}

\begin{figure}
\begin{tikzpicture}[auto,>=latex']
    \tikzstyle{block} = [draw, shape=rectangle, minimum height=3em, minimum width=3em, node distance=2cm, line width=1pt]
    \tikzstyle{block2} = [draw, shape=rectangle, minimum height=3em, minimum width=3em, node distance=1.3cm, line width=1pt]
    \tikzstyle{phant} = [shape=rectangle, minimum height=3em, minimum width=3em, node distance=2cm, line width=1pt]
  \tikzstyle{phant2} = [shape=rectangle, minimum height=0em, minimum width=3em, node distance=1.2cm, line width=1pt]
    \tikzstyle{dots} = [shape=circle,minimum height=1em, minimum width=2em, node distance=1cm]
    \tikzstyle{sum} = [draw, shape=circle, node distance=1.5cm, line width=1pt, minimum width=1.25em]
    \tikzstyle{branch}=[fill,shape=circle,minimum size=4pt,inner sep=0pt]
    \node at (-1.5,0) (input) {$x$};
    \node [block] (G1) {$\sqrt{\frac{d}{n}\Gamma}$};
    \node at (-2.5,0) [block, right of=G1] (T1) {$T_1^*$};
    \node [dots, below of=T1] (D1) {$\;\;\vdots\;\;$};
     \node [block, right of=T1] (H1) {$\frac{1}{\|\Gamma^{1/2}T_1\|}$};
       \node [dots, below of=H1] (D2) {$\quad\vdots\quad$};
         \node [block, below of=H1] (Hj) {$\frac{1}{\|\Gamma^{1/2}T_j\|}$};
         \node [dots, below of=Hj] (Hn) {$\quad\vdots\quad$};
           \node [block, below of=Hj] (Hnd) {$\frac{1}{\|\Gamma^{1/2}T_n\|}$};
             \node [phant, right of=H1] (P1) {};
                         \node [phant, right of=D2] (P2) {};
                                                  \node [phant, right of=Hn] (P3) {};
                                                  			    \node [phant, right of=Hn] (Pn) {};
			    \node [phant, right of=Hnd] (Pnd) {};
      \node [block, right of=P1] (TT1) {$T_1$};
        \node [block, right of=TT1] (HH1) {$\frac{1}{\|\Gamma^{1/2}T_1\|}$};
         
                    \node [dots, below of=HH1] (HH2) {$\;\;\vdots\;\;$};
    \node [sum, right of=HH1] (sum) {};
    \node [block2, right of=sum] (G2) {$\sqrt{\frac{d}{n}\Gamma}$};

          \node [phant, right of=Hj] (Pj) {\begin{turn}{-90} processing  level\end{turn}};
    
    \node at (sum) (plus) {{\footnotesize$+$}};
     
    \node at (14.3,0) (output) {$\tilde{x}$};
    \path (G1) -- coordinate (med) (T1);
    \node [block, below of=T1] (Tj) {$T_j^*$};
      \node [dots, below of=Tj] (Tn) {$\;\;\vdots\;\;$};
       \node [block, below of=Tj] (Tnd) {$T_n^*$};    
           \node [phant2, below of=Tnd] (END1) {signal analysis};     
                  \node [dots, below of=TT1] (TT2) {$\;\;\vdots\;\;$};    
       \node [block, below of=TT1] (TTj) {$T_j$};    
          \node [dots, below of=TTj] (TTn) {$\;\;\vdots\;\;$};    
       \node [block, below of=TTj] (TTnd) {$T_n$};    
        \node [block, right of=TTj] (HHj) {$\frac{1}{\|\Gamma^{1/2}T_j\|}$};
                            \node [dots, below of=HHj] (HHn) {$\;\;\vdots\;\;$};
                                    \node [block, right of=TTnd] (HHnd) {$\frac{1}{\|\Gamma^{1/2}T_n\|}$};
                                       \node [phant2, below of=HHnd] (END2) {signal synthesis};     
    \begin{scope}[line width=1pt]
         \draw[->] (input) -- (G1);
         \draw[->] (G1) -- (T1);
           \draw[->] (Tj) -- (Hj);
                \draw[->] (D2) -- (P2);
                     \draw[->] (H1) -- (P1);
                      \draw[->] (Hn) -- (Pn);
                      \draw[->] (Hnd) -- (Pnd);                     
     \draw[->] (Hj) -- (Pj);
    \draw[->] (Tnd) -- (Hnd);
            \draw[->] (G2) -- (output);
                   \draw[->] (P1) -- (TT1);    
                            \draw[->] (TT1) -- (HH1);    
                   
                     \draw[->] (P2) -- (TT2);
                   \draw[->] (Pj) -- (TTj);    
                                      \draw[->] (Pnd) -- (TTnd);    
                                                                            \draw[->] (Pn) -- (TTn);    
                           
           \draw[->] (med) node[branch] {} |- (Tj);
               \draw[->] (med) node[branch] {} |- (D1);
             \draw[->] (med) node[branch] {} |- (Tn);
                \draw[->] (med) node[branch] {} |- (Tnd);
                 \draw[->] (T1) -- (H1);
                 
                          \draw[->] (TTj) -- (HHj);
			 \draw[->] (TTnd) -- (HHnd);  
			 	 \draw[->] (HH1) -- (sum);    
                 \draw[->] (HHj) -| (sum);
                  \draw[->] (HH2) -| (sum);
                 \draw[->] (HHn) -| (sum);
                  \draw[->] (HHnd) -| (sum);
                          \draw[->] (sum) -- (G2);   
                 
    \end{scope}
\end{tikzpicture}
\caption{Analysis and synthesis scheme, in which fast matrix vector multiplication of $T_j^*$ and $T_j$ can still be used after and before pre- and postmultiplication with $\Gamma^{1/2}$, respectively.  }\label{fig:filtering}
\end{figure}
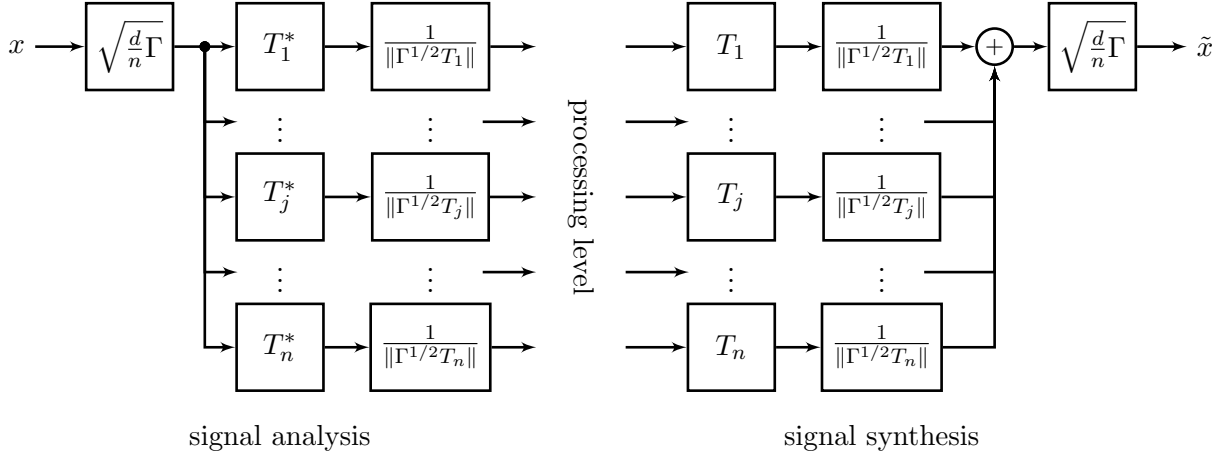

Note that \eqref{eq:tightness at the beginning}  is equivalent to 
\begin{equation}\label{eq:it motive}
\Gamma = \big(\frac{d}{n} \sum_{j=1}^n \frac{T_jT_j^*}{\trace(T_j^*\Gamma T_j)}\big)^{-1}.
\end{equation}
Thus, $\Gamma$ is the inverse of a generalized frame operator. 
Since this equation is invariant under scalings, we can look for a solution $\Gamma$ with $\trace(\Gamma)=1$. Let $\mathcal{P}$ be the collection of hermitian positive definite matrices in $\K^{d\times d}$ and denote by $\mathcal{P}_1$ the same space with the additional requirement that the trace is $1$. The fixed point equation \eqref{eq:it motive} gives rise to an iterative scheme that was already considered in \cite{Kent:1988kx,Tyler:1987fk} for $r=1$ to estimate the covariance of elliptical distributions. As initialization we choose $\Gamma_0=\frac{1}{d}I_d\in\mathcal{P}_1$ and define 
\begin{equation}\label{eq:iteration 0}
\Gamma_{k+1} := \frac{\big(\frac{d}{n} \sum_{j=1}^n \frac{T_jT_j^*}{\trace(T_j^*\Gamma_k T_j)}\big)^{-1} }{\trace\big((\frac{d}{n} \sum_{j=1}^n \frac{T_jT_j^*}{\trace(T_j^*\Gamma_k T_j)})^{-1}\big)} .
\end{equation}
Note that $\Gamma_1=S^{-1}/\trace(S^{-1})$, and to verify convergence, we shall follow the ideas of the technical  procedure used in \cite{Kent:1988kx,Tyler:1987fk} for $r=1$. For analysis purposes, we shall introduce the mapping 
\begin{equation}\label{eq:M}
M:\mathcal{P}\rightarrow \mathcal{P},\quad  M(\Gamma):=\frac{d}{n} \sum_{j=1}^n \frac{\Gamma^{1/2}T_j T^*_j\Gamma^{1/2}}{\trace(T^*_j \Gamma T_j)}, \quad\quad M_k:=M(\Gamma_k),
\end{equation}
so that
\begin{equation}\label{eq:def Gamma}
{\Gamma}_{k+1}  = \frac{{\Gamma}_k^{1/2}M_k^{-1}{\Gamma}_k^{1/2}}{\trace({\Gamma}_kM_k^{-1})},\qquad \Gamma_0:=\frac{1}{d}I_d.
\end{equation}
We shall first check that the mapping $M$ is injective up to scalings, which generalizes \cite[Theorem 2.1]{Tyler:1987fk} from $r=1$ to the general case:
\begin{lemma}\label{lemma:M}
Let $\{T_j\}_{j=1}^n\subset\K^{d\times r}$ be a $g$-frame and $\Gamma_a,\Gamma_b\in\K^d$ be positive definite. Then $M(\Gamma_a)=M(\Gamma_b)$ if and only if there is a positive constant $c>0$ such that $\Gamma_a=c\Gamma_b$. 
\end{lemma}
\begin{proof}
Without loss of generality, we can assume that $\Gamma_b=I_d$. Otherwise, replace $\{T_j\}_{j=1}^n$ with $\{\Gamma_b^{-1/2}T_j\}_{j=1}^n$. Let $\gamma_1$ be the largest eigenvalue of $\Gamma_a$ and $P_1$ be the associated eigenprojector. Moreover, let $\{P_i\}_{i=2}^s$ be one-dimensional eigenprojectors of $\Gamma_a$ associated to eigenvalues $\{\gamma_i\}_{i=2}^s$ such that $\gamma_1>\gamma_i$, for $i=2,\ldots, s$, and $\sum_{i=1}^s \gamma_i P_i =\Gamma_a$. Note that $\gamma_2,\ldots,\gamma_s$ do not need to be pairwise distinct. Since $\Gamma_a^{1/2}P_1\Gamma^{1/2}_a=\gamma_1P_1$ and $\sum_{i=1}^s P_i=I_d$, we obtain
\begin{align*}
\trace(P_1M(\Gamma_a)) & = \frac{d}{n}\sum_{j=1}^n \frac{\trace(P_1 \Gamma_a^{1/2}T_j T^*_j\Gamma_a^{1/2})}{\trace(T^*_j \Gamma_a T_j)}  \\
& = \gamma_1 \frac{d}{n}\sum_{j=1}^n \frac{\trace(T_j^* P_1 T_j)}{\trace(T^*_j \Gamma_a T_j)}  \\
& = \gamma_1\frac{d}{n}\sum_{j=1}^n \frac{\trace(T_j^* P_1 T_j)}{\sum_{i=1}^s\gamma_i\trace(T^*_j P_i T_j)} \\ 
&\geq   \gamma_1 \frac{d}{n}\sum_{j=1}^n \frac{\trace(T_j^* P_1 T_j)}{\gamma_1\trace(T^*_j T_j)} = \trace(P_1 M(I_d)).
\end{align*}
Now, $M(\Gamma_a)=M(I_d)$ implies that either $\trace(T_j P_i T^*_j)=0$, for all $j=1,\ldots,n$ and $i=2,\ldots,r$ or that $P_1=I_d$ and, hence,  $\Gamma_a=\gamma_1I_d$. Since $\{T_j\}_{j=1}^n$ is a generalized frame, $\trace(T^*_j P_i T_j)$ cannot vanish simultaneously for all $j=1,\ldots,n$, so that, indeed, $\Gamma_a$ is a nonzero multiple of the identity. 
\end{proof}
The following result says that if we find a proper tight $g$-frame with unit norm based on \eqref{eq:tight def apporach}, then this tight $g$-frame is unique:
\begin{proposition}
Let $\{T_j\}_{j=1}^n$ be a $g$-frame and suppose that there are two positive definite matrices $\Gamma_a$ and $\Gamma_b$ such that both $\{\Gamma_a^{1/2}T_j/\|\Gamma_a^{1/2}T_j\|_{HS}\}_{j=1}^n$ and $\{\Gamma_b^{1/2}T_j/\|\Gamma_b^{1/2}T_j\|_{HS}\}_{j=1}^n$ are tight. Then those two tight $g$-frames are identical. 
\end{proposition}
\begin{proof}
The tightness assumptions imply that $M(\Gamma_a)=M(\Gamma_b)=I$. According to Lemma \ref{lemma:M}, there is a positive constant $c$ such that $\Gamma_a=c\Gamma_b$. Therefore, the two tight $g$-frames are identical. 
\end{proof}


Next, we use the scheme \eqref{eq:def Gamma} to compute a unit norm tight  $g$-frame:
\begin{theorem}\label{theorem:Paul}
Let $\{T_j\}_{j=1}^n\subset\K^{d\times r}$ satisfy the following points:
\begin{enumerate}[(i)]
\item \label{it:1} $\{T_j\}_{j=1}^n$ is a  g-frame.
\item \label{it:2} If $\{1,\ldots,n\}=N_1\cup N_2$ and $\N_1\cap N_2=\emptyset$, then $\bigcup_{j\in N_i}\range(T_j)$ spans $\K^d$ for either $i=1$ or $i=2$.
\item \label{it:3} If $L$ is a proper linear subspace of $\K^d$, then $\#\{j: \range(T_j)\subset L\} < \frac{n}{d}$.
\item \label{it:4} If $L$ is a proper linear subspace of $\K^d$, then $\#\{j: \range(T_j)\subset L\} < \dim(L)/\alpha$, where $\alpha:=\max_{1\leq j\leq n}\trace(T_j^* S^{-1}T_j)$ and $S$ is the $g$-frame operator of $\{T_j\}_{j=1}^n$.
\end{enumerate}
%
%
%
%
Then the recursive scheme \eqref{eq:def Gamma} with $\Gamma_0=\frac{1}{d}I_d$ converges towards a positive definite $\Gamma$ and $\{R_j\}_{j=1}^n$ defined by 
\begin{equation}\label{eq:R}
R_j:=\frac{\Gamma^{1/2}T_j}{\|\Gamma^{1/2}T_j\|_{HS}}
\end{equation}
is a tight  g-frame. 
\end{theorem}
This theorem generalizes results in \cite{Kent:1988kx,Tyler:1987fk}, where convergence is verified for $r=1$. The conditions in Theorem \ref{theorem:Paul} are redundant. Condition \eqref{it:2} clearly implies \eqref{it:1}. For $d\geq 2$, \eqref{it:3} yields \eqref{it:2}. In fact, the conditions \eqref{it:1}, \eqref{it:2}, and \eqref{it:3} depend on the range of each $T_j$ but not on their norm. Note that condition \eqref{it:3} can only be satisfied by some $\{T_j\}_{j=1}^n$ if $\frac{n}{d}>\lfloor\frac{d-1}{r}\rfloor$, which yields $n>d\lfloor\frac{d-1}{r}\rfloor$. Condition \eqref{it:4} is independent of global scalings since multiplication of all $T_j$ with some constants $c$ means that the inverse frame operator needs to be divided by $c^2$. It requires $\alpha < r$,  which is, in fact, quite weak:
\begin{proposition}\label{prop:alpha r}
Suppose that $\{T_j\}_{j=1}^n\subset\K^{d\times r}$ is a g-frame. We then have $\trace(T_j^*S^{-1}T_j)\leq r$, for all $j=1,\ldots,n$, and if there is $j$ with $\trace(T_j^*S^{-1}T_j) = r$, then $T_j$ does not have any zero columns and $\range(S^{-1/2}T_k)\perp \range(S^{-1/2}T_j)$, for all $k\neq j$.
\end{proposition}
\begin{proof}[Proof of Proposition \ref{prop:alpha r}]
We first choose an orthonormal basis $\{e_i\}_{i=1}^r$ for $\K^r$ and define, for some $1\leq j\leq n$, the index set $\mathcal{I}_j:=\{i:1\leq i\leq r,\; S^{-1/2}T_j e_i\neq 0 \}$. We apply Proposition \ref{th:Parse} $r$-times to derive
\begin{align*}
r & \geq  \sum_{i\in\mathcal{I}_j} \sum_{k=1}^n \big\|T_k^* S^{-1/2} \frac{S^{-1/2} T_j e_i}{\|S^{-1/2} T_j e_i\|} \big\|^2\\
 & =  \sum_{i\in\mathcal{I}_j} \frac{1}{\|S^{-1/2} T_j e_i\|^2}\sum_{j=1}^n \sum_{l=1}^r \big|\langle  S^{-1/2}T_k e_l,  S^{-1/2} T_j e_i \rangle \big|^2\\
 & \geq   \sum_{i\in\mathcal{I}_j}  \frac{1}{\|S^{-1/2} T_j e_i\|^2} \big|\langle  S^{-1/2}T_j e_i,  S^{-1/2} T_j e_i \rangle \big|^2\\
 & =   \sum_{i\in\mathcal{I}_j} \|S^{-1/2} T_j e_i\|^2 = \|S^{-1/2} T_j\|^2_{HS} = \trace(T_j^* S^{-1} T_j).\qedhere
\end{align*}
The assumption $\trace(T_j^*S^{-1}T_j) = r$ implies that the two above inequalities become equalities, which yield the required statements. 
\end{proof}
Note that Proposition \ref{prop:alpha r} bounds the worst case scenario. Since $I=S^{-1}S$, taking the trace on both sides yields that $d=\sum_{j=1}^n \trace(T_j^* S^{-1}T_j)$. If $\{T_j\}_{j=1}^n$ has unit norm and  is close to being tight meaning $S\approx \frac{n}{d}I$, then $\alpha\approx \frac{d}{n}$. If $\{T_j\}_{j=1}^n$ are sufficiently generic or in sufficient general position, then \eqref{it:1}-\eqref{it:4} are satisfied for sufficiently large $n$:
\begin{remark}
If $\{T_j\}_{j=1}^n$ is a sample of a continuous distribution on $\K^{d\times r}$ and $n$ is sufficiently large, then with probability one all of the assumptions in Theorem \ref{theorem:Paul} are satisfied. 
\end{remark}

\begin{proof}[Proof of Theorem \ref{theorem:Paul}]
Since $\{T_j\}_{j=1}^n$ is a  g-frame, the sequence $\{\Gamma_k\}_{k=1}^\infty$ is well-defined. It is also clear that $\Gamma_{k}$ is hermitian positive definite and $\trace(\Gamma_{k})=1$, for all $k=0,1,2,\ldots$. If we suppose that $\{\Gamma_k\}_{k=1}^\infty$ converges towards $\Gamma$, then $M(\Gamma)=I_d$ must hold, and a direct computation yields that $\{R_j\}_{j=1}^n$ is a tight  g-frame with $\|R_j\|_{HS}=1$, for $j=1,\dots,n$. 

It remains to verify convergence, which we check in two steps:\\
\noindent\textbf{Step 1)} (refers to Lemma 2.1 in \cite{Tyler:1987fk}) Let $\lambda_{1,k}$ and $\lambda_{d,k}$ be the largest and smallest eigenvalues of $M_k$, respectively. We observe 
\begin{align*}
I &= \frac{d}{n}\sum_{j=1}^n \frac{M^{-1/2}_k\Gamma^{1/2}_kT_j T^*_j\Gamma^{1/2}_kM^{-1/2}_k}{\trace(T^*_j \Gamma_k T_j)}.\\
\intertext{The positive definite square root of $\Gamma^{1/2}_{k}M_k^{-1}\Gamma^{1/2}_{k}$ has the form $Q_kM_k^{-1/2}\Gamma_k^{1/2}$, where $Q_k$ is an orthogonal matrix. Since $Q_kQ^*_k=I$, we obtain}
I & = \frac{d}{n}\sum_{j=1}^n  \frac{Q_k M^{-1/2}_k\Gamma^{1/2}_kT_j T^*_j\Gamma^{1/2}_kM^{-1/2}_kQ^*_k}{\trace(T^*_j \Gamma_k T_j)}.
\end{align*}
According to \eqref{eq:def Gamma}, we have 
\begin{align*}
{\Gamma}^{1/2}_{k+1}=\frac{(\Gamma^{1/2}_{k}M_k^{-1}\Gamma^{1/2}_{k})^{1/2}}{\sqrt{\trace({\Gamma}_kM_k^{-1})}} = \frac{Q_kM_k^{-1/2}\Gamma_k^{1/2}}{\sqrt{\trace({\Gamma}_kM_k^{-1})}},
\end{align*}
which yields
\begin{align*}
I & =  \frac{d}{n}\sum_{j=1}^n \big(\frac{\Gamma^{1/2}_{k+1}T_j T^*_j\Gamma^{1/2}_{k+1}}{\trace(T^*_j \Gamma_{k} T_j)}\big) \trace({\Gamma}_kM_k^{-1}) \\
& =  \frac{d}{n}\sum_{j=1}^n \big(\frac{\Gamma^{1/2}_{k+1}T_j T^*_j\Gamma^{1/2}_{k+1}}{\trace(T^*_j \Gamma_{k+1} T_j)}\big) \frac{\trace(T^*_j \Gamma_{k+1} T_j)\trace({\Gamma}_kM_k^{-1}) }{\trace(T^*_j \Gamma_{k} T_j)}.
\intertext{According to \eqref{eq:def Gamma}, we have the identity $\trace(T^*_j \Gamma_{k+1} T_j)=\frac{\trace(T_j^* \Gamma_k^{1/2} M_k^{-1}\Gamma_k^{1/2} T_j)}{\trace({\Gamma}_kM_k^{-1}) }$, so that }
I & = \frac{d}{n}\sum_{j=1}^n \big(\frac{\Gamma^{1/2}_{k+1}T_j T^*_j\Gamma^{1/2}_{k+1}}{\trace(T^*_j \Gamma_{k+1} T_j)}\big) \frac{\trace(T_j^* \Gamma_k^{1/2} M_k^{-1}\Gamma_k^{1/2} T_j)}{\trace(T^*_j \Gamma_{k} T_j)}\\
& = \frac{d}{n}\sum_{j=1}^n \big(\frac{\Gamma^{1/2}_{k+1}T_j T^*_j\Gamma^{1/2}_{k+1}}{\trace(T^*_j \Gamma_{k+1} T_j)}\big) \frac{\| M^{-1/2}_k \Gamma^{1/2}_k T_j\|^2_{HS}}{\| \Gamma_k^{1/2}T_j\|^2_{HS}}
\end{align*}
holds. 
Since each of the matrices $\frac{\Gamma^{1/2}_{k+1}T_j T^*_j\Gamma^{1/2}_{k+1}}{\trace(T^*_j \Gamma_{k+1} T_j)}$ is positive-semi-definite, the definition of $M_{k+1}$ in \eqref{eq:M} with its largest and smallest eigenvalues implies 
\begin{equation*}
\lambda^{-1}_{d,k}M_{k+1}\geq I \geq \lambda^{-1}_{1,k}M_{k+1}.
\end{equation*}
The left inequality yields $\lambda_{d,k+1}/\lambda_{d,k}\geq 1$ and the right inequality implies $ \lambda_{1,k+1}/ \lambda_{1,k}\leq 1$. Thus, as required, the sequence $(\lambda_{d,k})_{k=1}^\infty$ is increasing, $(\lambda_{1,k})_{k=1}^\infty$ is decreasing, and both converge towards $\lambda_d\leq 1$ and $\lambda_1\geq 1$, respectively. 

%
%
%
%
%


\noindent\textbf{Step 2)} (refers to Theorem 2.2 and Corollary 2.2 in \cite{Tyler:1987fk})
Since $\Gamma_k$ is positive definite with trace $1$, there is a subsequence $(\Gamma_{k_m})_{m=1}^\infty$ that converges towards some positive semi-definite matrix $\Gamma$. We must now verify that $\Gamma$ is positive definite and that the entire sequence converges. 

If $\Gamma T_j = 0$, then let $\theta_j\in\K^{d\times r}$ be such that 
\begin{equation*}
\Gamma_{k_m}^{1/2}T_j/\trace(T_j^*\Gamma_{k_m}T_j)\rightarrow \theta_j.
\end{equation*}
For $\mathcal{J}:=\{j: \Gamma T_j\neq 0 \}$, we observe that the subsequence $(M_{k_m})_{m=1}^\infty$ converges to 
\begin{equation}\label{eq:M2}
M  = \frac{d}{n} \big(\sum_{j\in \mathcal{J}} \frac{\Gamma^{1/2}T_j T^*_j\Gamma^{1/2}}{\trace(T^*_j \Gamma T_j)}+\sum_{j\not\in \mathcal{J}} \theta_j\theta_j^*\big).
\end{equation}
Step 1 ensures that $M$ is invertible since its smallest eigenvalue is positive.  Let $Q_{k_m}$ be the orthogonal matrices in Step 1, so that $Q_{k_m}\rightarrow Q$ and $(\Gamma^{1/2}M^{-1}\Gamma^{1/2})^{1/2} = QM^{-1/2}\Gamma^{1/2}$ is satisfied.  
According to the definition \eqref{eq:def Gamma}, $(\Gamma_{k_m+1})_{m=1}^\infty$ converges towards 
\begin{equation*}
\Gamma_0=\Gamma^{1/2}M^{-1}\Gamma^{1/2}/\trace(M^{-1}\Gamma)
\end{equation*}
and a short calculation yields that the sequence $(M_{k_{m}+1})_{m=1}^\infty$ converges to
\begin{equation*}
M_0  = \frac{d}{n} \big(\sum_{j\in \mathcal{J}} \frac{\Gamma_0^{1/2}T_j T^*_j\Gamma_0^{1/2}}{\trace(T^*_j \Gamma_0 T_j)}+\sum_{j\not\in \mathcal{J}} \phi_j\phi_j^*\big),
\end{equation*}
where $\phi_j=QM^{1/2}\theta_j/\sqrt{\trace(\theta_j^*M^{-1}\theta_j)}$.

Manipulations as in Step 1 applied to the formulas for $M$ and $M_0$ imply that 
\begin{equation}\label{eq:1}
I_d = \frac{d}{n}\sum_{j\in\mathcal{J}} \frac{M^{-1/2}\Gamma^{1/2}T_j T^*_j\Gamma^{1/2}M^{-1/2}}{\trace(T^*_j \Gamma T_j)} + \sum_{j\not\in\mathcal{J}}M^{-1/2}\theta_j\theta_j^*M^{-1/2}
\end{equation}
as well as 
\begin{equation}\label{eq:2}
I_d = \frac{1}{\lambda_1}M_0+\frac{d}{n}\sum_{j\in\mathcal{J}} \frac{\Gamma_0^{1/2}T_j T^*_j\Gamma_0^{1/2}}{\trace(T^*_j \Gamma_0 T_j)} \big(\frac{\trace(T^*_j\Gamma^{1/2}M^{-1}\Gamma^{1/2}T_j)}{\trace(T^*_j \Gamma T_j)}-\frac{1}{\lambda_1}\big) 
+\frac{d}{n}\sum_{j\not\in\mathcal{J}}\phi_j\phi_j^*\big(\trace(\theta_j^*M^{-1}\theta_j-\frac{1}{\lambda_1})\big).
\end{equation}

According to Step 1, the largest and smallest eigenvalue of both $M$ and $M_0$ are $\lambda_1$ and $\lambda_d$, respectively. Let $P$ and $P_0$ be the eigenprojectors of $M$ and $M_0$, respectively, associated with $\lambda_1$ and $s=\rank(P)$ and $s_0=\rank(P_0)$. As in \cite{Tyler:1987fk}, without loss of generality, we can suppose $s_0\geq s$. 

By multiplying both sides from the left and the right with $P_0$, the relations $P_0^2=P_0$ and $\frac{1}{\lambda_1}P_0MP_0=P_0$ yield
\begin{equation*}
0 = \sum_{j\in\mathcal{J}} \frac{P_0\Gamma_0^{1/2}T_j T^*_j\Gamma_0^{1/2}P_0}{\trace(T^*_j \Gamma_0 T_j)} \big(\frac{\trace(T^*_j\Gamma^{1/2}M^{-1}\Gamma^{1/2}T_j)}{\trace(T^*_j \Gamma T_j)}-\frac{1}{\lambda_1}\big) 
+\sum_{j\not\in\mathcal{J}}P_0\phi_j\phi_j^*P_0\big(\trace(\theta_j^*M^{-1}\theta_j-\frac{1}{\lambda_1})\big).
\end{equation*}

Thus, for $j\in\mathcal{J}$, we either have 
\begin{equation*}
P_0\Gamma_0^{1/2}T_j T^*_j\Gamma_0^{1/2}P_0=0 \quad\text{or}\quad \frac{\trace(T^*_j\Gamma^{1/2}M^{-1}\Gamma^{1/2}T_j)}{\trace(T^*_j \Gamma T_j)}=\frac{1}{\lambda_1}.
\end{equation*} 
The first option yields $P_0QM^{-1/2}\Gamma^{1/2}T_j=0$. The second option implies $\langle P,\Gamma^{1/2}T_jT_j^* \Gamma^{1/2}\rangle=\langle I,\Gamma^{1/2}T_jT_j^* \Gamma^{1/2}\rangle$, which yields after some computations $P\Gamma^{1/2}T_j = \Gamma^{1/2}T_j$.

For $j\notin\mathcal{J}$, we obtain 
\begin{equation*}
P_0\phi_j\phi_j^*P_0=0\quad\text{or}\quad \trace(\theta_j^*M^{-1}\theta_j-\frac{1}{\lambda_1})=0.
\end{equation*}
Similar to the above considerations, the first option yields $P_0QM^{-1/2}\theta_j = 0$. The second option implies $P\theta_j = \theta_j$.

We now premultiply both sides of \eqref{eq:1} with $P_0Q$ and postmultiply by $I-P$. The above four options and using that $P$ and $M^{-1/2}$ commute imply $P_0Q(I-P)=0$, which is equivalent to $Q^*P_0Q = Q^*P_0QP$. Since $s_0\geq s$, we obtain $P=Q^*P_0Q$, so that $QP=P_0Q$. The latter implies with the above that for $j\in\mathcal{J}$, either 
\begin{equation*}
P\Gamma^{1/2}T_j=0\quad\text{or}\quad P\Gamma^{1/2}T_j=\Gamma^{1/2}T_j.
\end{equation*}
Hence, we can split $\{1,\ldots,n\}$ into two disjoint index sets $N_1$ and $N_2$ such that 
\begin{equation*}
P\Gamma^{1/2}T_j =
\begin{cases}
0,& j\in N_1\\
\Gamma^{1/2}T_j,& j\in N_2.
\end{cases}
\end{equation*}
Condition \eqref{it:2} yields that $\spann(\range(T_j):j\in N_i\})$ equals $\K^d$ for either $i=1$ or $i=2$. If this holds for $i=1$, then we must have $P\Gamma^{1/2}=0$. If it holds for $i=2$, then we derive $P\Gamma^{1/2}=\Gamma^{1/2}$. 

Suppose now that $P\Gamma^{1/2}=0$ holds. The same arguments as in the previous paragraph yield, for $j\in\mathcal{J}$, that either $P\theta_j=0$ or $P\theta_j=\theta_j$, see also \cite{Tyler:1987fk}. Pre- and postmultiplying both sides in \eqref{eq:M2} by $P$ yields 
\begin{equation*}
\lambda_1 P = \frac{d}{n} \sum_{j\in \mathcal{J}_0} \theta_j\theta_j^*,
\end{equation*}
where $\mathcal{J}_0:=\{j : j\not\in\mathcal{J},\; P\theta_j=\theta_j\}$. Next, we take the trace on both sides and use that $\trace(\theta_j^*\theta_j)=1$ to derive
\begin{equation*}
\lambda_1 s \leq \frac{d}{n} n_1,
\end{equation*}
where $n_1$ is the number of $T_j$ whose range is contained in the null space of $\Gamma$. Condition \eqref{it:3} yields $\lambda_1\leq \frac{dn_1}{sn} <1$, which is a contradiction to the results of Step 1. Thus, we must have $P\Gamma^{1/2}=\Gamma^{1/2}$, so that
\begin{equation}\label{eq:2.10}
P = \frac{d}{n\lambda_1} \sum_{j\in\mathcal{J}} \frac{\Gamma^{1/2}T_jT_j^*\Gamma^{1/2}}{\trace(T_j^*\Gamma T_j)} + \frac{d}{n\lambda_1} \sum_{j\in\mathcal{J}_0}\theta_j\theta_j^*.
\end{equation}
Since the ranks of the two summations in \eqref{eq:M2} are additive, see also \cite{Tyler:1987fk}, the ranks of the two summations in \eqref{eq:2.10} are additive. Hence, the two terms themselves must be orthogonal projections. According to condition \eqref{it:1}, the rank of the first term equals $\rank(\Gamma)$. If $\rank(\Gamma)<s$, then taking the trace of the second term implies with condition \eqref{it:3} that $\lambda_1<\frac{1}{s-\rank(\Gamma)}\leq 1$, which is a contradiction to Step 1. Therefore, $\mathcal{J}_0$ is empty and $s=\rank(\Gamma)$. Taking the trace of the first term in \eqref{eq:2.10} yields
\begin{equation}\label{eq:lambda}
\lambda_1 = \frac{d (n-n_1)}{n s}.
\end{equation}

We obtain $d=\trace(M)\geq s\lambda_1+(d-s)\lambda_d$, so that \eqref{eq:lambda} implies
\begin{equation}\label{eq:last eq}
\lambda_d \leq \frac{dn_1}{n(d-s)}.
\end{equation}

At this point, we claim that the assumption \eqref{it:4} implies for at least one $k$ that
\begin{equation}\label{eq:main condition 2}
\lambda_{d,k} > \frac{d}{n} \frac{\#\{j: \range(T_j)\subset L\}}{\dim(L)}
\end{equation}
holds, for all proper linear subspaces $L\subset \K^d$, but postpone the verification to the end of this proof.

Since $\lambda_{d,k}$ is an increasing sequence, \eqref{eq:main condition 2} implies $d\frac{n_1}{n(d-s)}<\lambda_d$ if $d>s$. This violates \eqref{eq:last eq}, so that $M$ and hence $\Gamma$ must have full rank and, therefore, $\Gamma$ is positive definite. Also, $P$ must have full rank implying $\lambda_1=\lambda_d=1$ and $M=I$. Since the eigenvalues are monotone, the entire sequence $(M_k)_{k=1}^\infty$ converges towards $I$. The latter can be used with Banach's Fix point theorem to verify that also $(\Gamma_k)_{k=1}^\infty$ must converge, hence, towards $\Gamma$. By continuity, we obtain $M(\Gamma)=I$.

We still need to verify \eqref{eq:main condition 2}. We observe $\Gamma_1=S^{-1}/\trace(S^{-1})$ and define $W_j:=\Gamma_1^{1/2}T_j$, which yields $\sum_{j=1}^n W_j W_j^* = I/\trace(S^{-1})$. By using $\alpha=\max_{1\leq j\leq n}\trace(T_j^* S^{-1}T_j)$ as in \eqref{it:4}, this implies $M_1\succeq \frac{d}{n\alpha}I$, so that
\begin{equation*}
\lambda_{d,1}\geq \frac{d}{n\alpha}>\frac{d}{n}\frac{\#\{j: \range(T_j)\subset L\}}{\dim(L)},
\end{equation*}
where the last inequality is due to \eqref{it:4}. This concludes the proof.
\end{proof}


\begin{remark}
The inversion of $M_k$ in the iterative scheme \eqref{eq:def Gamma} is numerically stable because there is a lower positive bound on the smallest eigenvalues of $M_k$. Therefore, we can expect that our scheme is quite stable overall. 
\end{remark}

Let us have a look at few pathological examples first:
\begin{example}
\begin{itemize}
\item[1)] 
If $\{T_j\}_{j=1}^n$ is already a unit norm tight  g-frame, then $\Gamma_1=\Gamma=\frac{1}{d}I_d$ and $\{R_j\}_{j=1}^n=\{T_j\}_{j=1}^n$.

\item[2)] 
If the g-frame consists of a single matrix $T\in\K^{d\times r}$, hence, $r=d$ and $T$ is regular, then $\Gamma_1=\Gamma=T T^*$ and \eqref{eq:R} yields $R=\frac{1}{\sqrt{d}}(TT^*)^{-1/2}T$, which is a unit norm tight g-frame. 
\end{itemize}
\end{example}

Next, we illustrate Theorem \ref{theorem:Paul} with few numerical examples:
\begin{example}\label{ex:some}
Let $d=2$, $r=1$, $n=3$. We pick $\alpha_1, \alpha_2, \alpha_3$ from a uniform distribution on $[0,2\pi]$ and define 
 $T_j=\binom{\cos(\alpha_j)}{\sin(\alpha_j)}$,  $j=1,2,3$. By multiplication with $-1$ and  rotation of all $3$ vectors, we can restrict the angles to lie between $0$ and $\frac{2}{3}\pi$. For each random choice $\{T_j\}_{j=1}^3$, we compute a unit norm tight frame $\{R_j\}_{j=1}^3$ using our proposed algorithm. Up to rotations $U_\theta= \begin{pmatrix}\cos(\theta) &-\sin(\theta)\\ \sin(\theta) & \cos(\theta)
\end{pmatrix}$ and multiplication by $-1$, there is only one single unit norm tight frame with three elements. We choose $\{Y_j\}_{j=1}^3$, where $Y_j=\binom{\cos(\beta_j)}{\sin(\beta_j)}$ and $\beta_1=0$, $\beta_2=1/3\pi$, and $\beta_3=2/3\pi$. Therefore, to find the tight frame with unit norm that is closest to $\{T_j\}_{j=1}^3$, we minimize the distance to $\{T_j\}_{j=1}^3$ over all rotations, i.e, 
\begin{equation*}
 \hat{\theta}:=\arg\min_{\theta\in [-2/3\pi,2/3\pi]} \sum_{j=1}^3 \| U_\theta Y_j-T_j\|^2,
 \end{equation*}
 and define the closest tight frame by $\{Z_j\}_{j=1}^3:=\{U_{\hat{\theta}} Y_j\}_{j=1}^3$. Note that we can suppress the multiplication by $-1$ because the angles of $\{T_j\}_{j=1}^3$ only run in $[0,\frac{2}{3}\pi]$. The average error of $\sum_{j=1}^n\|Z_j-R_j\|^2$ over $1000$ realizations is $\approx .0016=1/625$, see also Fig.~\ref{fig:Paul} for a visualization of few examples. In our numerical experiments, we observed that our proposed algorithm finds a tight frame that is almost identical to the closest tight frame if all pairs $T_i$ and $T_j$, for $i\neq j$, are far enough from each other.
 \begin{figure}
\centering
\includegraphics[width=.28\textwidth]{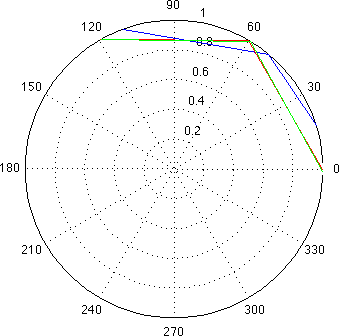}
\includegraphics[width=.28\textwidth]{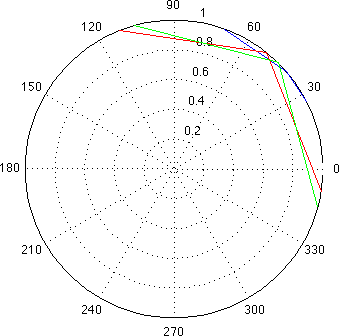}
\includegraphics[width=.28\textwidth]{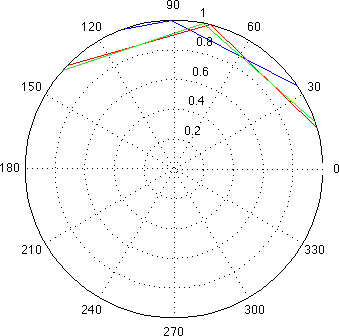}

\vspace{1ex}

\includegraphics[width=.28\textwidth]{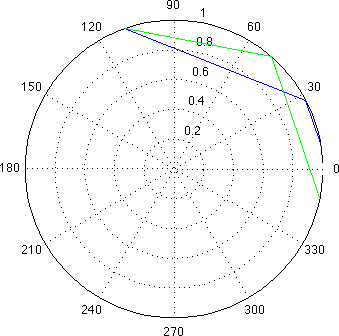}
\includegraphics[width=.28\textwidth]{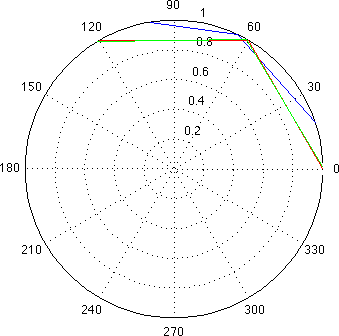}
\includegraphics[width=.28\textwidth]{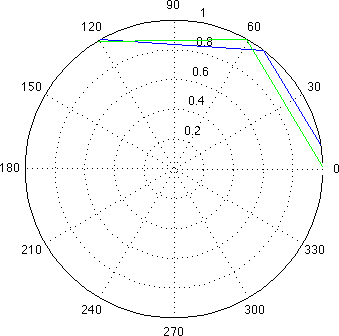}
\caption{Original frame $\{T_j\}_{j=1}^3$ in blue, optimal frame $\{Z_j\}_{j=1}^3=\{U_{\hat{\theta}} Y_j\}_{j=1}^3$ in green, and our proposed algorithm finds $\{\Gamma^{1/2}T_j/\|\Gamma^{1/2}T_j\|\}_{j=1}^3$ in red. (Green and red lines are sometimes right on top of each other).} \label{fig:Paul}
\end{figure}  
\end{example}

\begin{example}\label{ex:theta}
For $0\leq t\leq 1/2$, define 
\begin{equation*}
T_1(t)=\begin{pmatrix}
\sqrt{1-t} & 0 \\ 0 & \sqrt{t}\end{pmatrix},
\quad 
T_2(t)=\begin{pmatrix}  \sqrt{t}& \sqrt{t} \\ 0 & \sqrt{1-2t}
\end{pmatrix},\quad 
T_3(t)=\begin{pmatrix} \sqrt{\frac{1-t}{2}}& 0 \\ 0 & \sqrt{\frac{1+t}{2}}
\end{pmatrix},
\end{equation*}
and let $S(t)$ denote the associated g-frame operator. Note that $\{T_j(t)\}_{j=1}^3$ satisfies the assumptions of Theorem \ref{theorem:Paul}, for all $0\leq t\leq 1/2$. If $t=0$, then we have a tight generalized frame with unit norms. Our algorithm provides a Parseval g-frame $\{\sqrt{2/3}R_j(t)\}_{j=1}^3$ with equal HS-norm, see Fig.~\ref{fig:error} for the errors $\sum_{j=1}^3 \|T_j(t)- \sqrt{2/3}R_j(t)\|_{HS}^2$ and $\sum_{j=1}^3 \|T_j(t)- S(t)^{-1/2}T_j(t)\|_{HS}^2$.


\begin{figure}
\centering
\includegraphics[width=.35\textwidth]{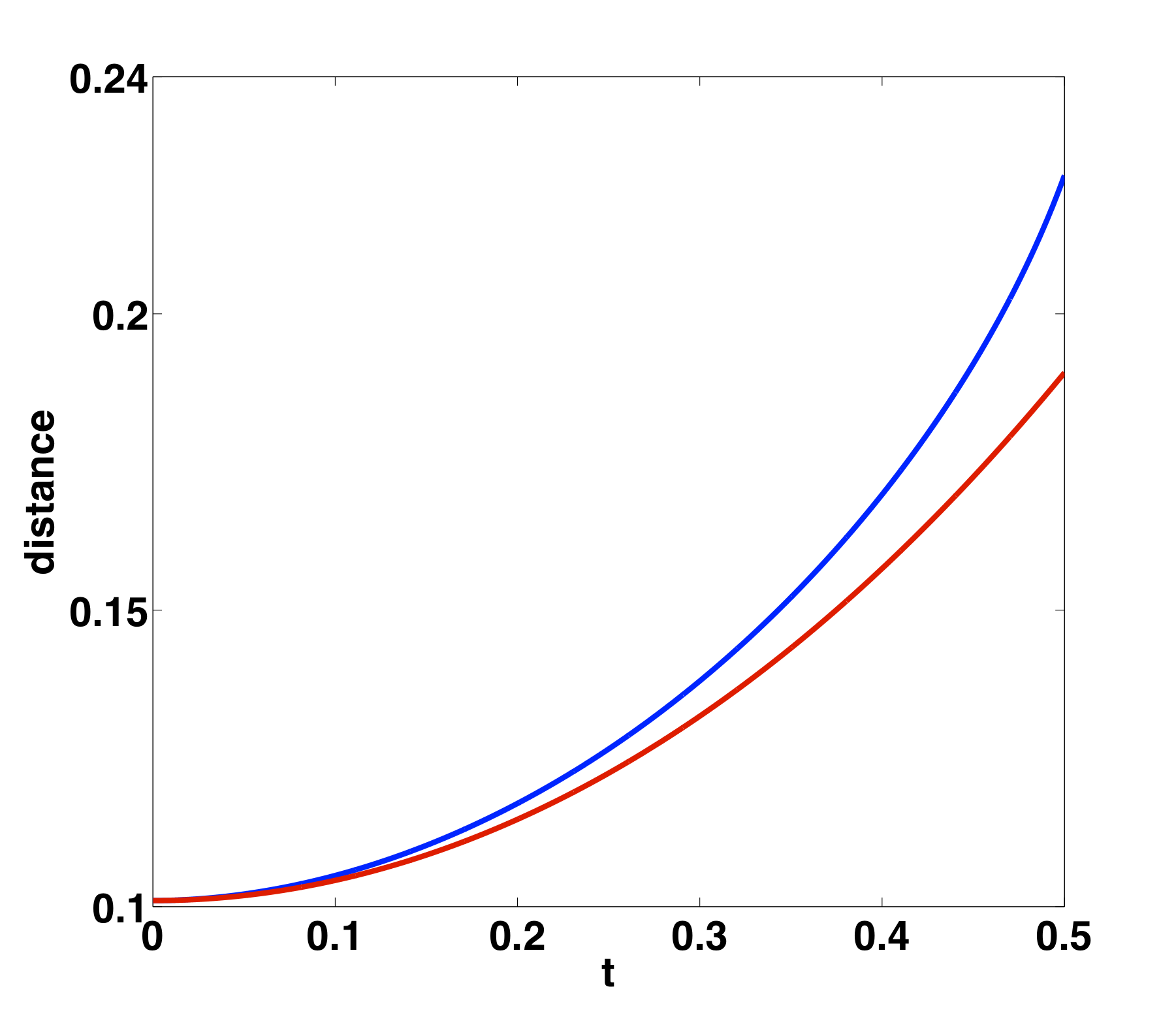}
\caption{$t$ versus the distance (blue) $\sum_{j=1}^3 \|T_j(t)- \sqrt{2/3}R_j(t)\|_{HS}^2$  and (red) $\sum_{j=1}^3 \|T_j(t)- S(t)^{-1/2}T_j(t)\|_{HS}^2$ in Example \ref{ex:theta}. Note that $\{T_j(0)\}_{j=1}^3$ is tight but needs rescaling to become the Parseval generalized frame $\{\sqrt{2/3}T_j(0)\}_{j=1}^3$. It is clear that $R_j(0)=T_j(0)$ and $\sqrt{2/3}T_j(0)=T_j(0)S(0)^{-1/2}$ holds, and we have $\sum_{j=1}^3 \|T_j(0)- \sqrt{2/3}T_j(0)\|_{HS}^2=0.1010$. The latter explains why the distance plots do not start at $0$.
 }\label{fig:error}
\end{figure}
\end{example}

\begin{example}
We choose each entry of each element in $\{T_j\}_{j=1}^3\subset\K^{2\times 2}$ independently according to a uniform distribution on $[0,1]$ and normalize so that $\|T_j\|_{HS}=1$. All numbers in the following are averaged over $10,000$ realizations, and let $S$ denote the generalized frame operator of  $\{T_j\}_{j=1}^3$. 
According to Theorem \ref{th:Parse}, $\{S^{-1/2}T_j\}_{j=1}^3$ is the Parseval generalized frame that is closest to $\{T_j\}_{j=1}^3$, and we compute $ \sum_{j=1}^3 \| T_j-S^{-1/2}T_j\|^2_{HS}\approx 0.61$. However, the elements of $\{S^{-1/2}T_j\}_{j=1}^3$ may not have equal norm. Based on Theorem \ref{theorem:Paul}, the collection $\{\tilde{R}_j\}_{j=1}^3 =\{\sqrt{2/3}R_j\}_{j=1}^3$ is a Parseval generalized frame,  
 its elements have equal norm, and we compute $ \sum_{j=1}^3 \| T_j-\tilde{R}_j\|^2_{HS}\approx 0.71$. Thus, the additional property of having equal norm costs $\approx 0.10=0.71-0.61$. It remains open though if there are other Parseval generalized frames whose elements have equal norm and that are closer to $\{T_j\}_{j=1}^3$.
\end{example}

Let us also illustrate when the algorithm fails to converge:
\begin{example}
For $t=0$, the collection $\{T_j(t)\}_{j=1}^3$, where 
\begin{equation*}
T_1(t)=\begin{pmatrix}  1 & 0\\ 0 &0
\end{pmatrix},\quad T_2(t)=\begin{pmatrix}  1 & t\\ 0 &0
\end{pmatrix},\quad\text{and}\quad T_3(t)=\begin{pmatrix}  0 & 0\\ 0 &1
\end{pmatrix}, 
\end{equation*}
violates the conditions in Theorem \ref{theorem:Paul}, and indeed, the iterative scheme does not converge towards a positive definite matrix $\Gamma$. For $t>0$, on the other hand, we observe convergence numerically.
\end{example}
The subsequent sections are dedicated to provide some examples of random samples satisfying the assumptions of Theorem \ref{theorem:Paul}. We shall also provide examples that allow for fast matrix vector multiplications such as convolution operators and we  support the intuition that $\Gamma$ is close to the identity if the sample is close to being tight. 

\section{Examples of random matrices satisfying the assumptions for convergence}\label{sec:II}
 We first fuse the concepts of generalized frames and probabilistic frames as developed in \cite{Sun:2006fk} and \cite{Okoudjou:2010aa}, respectively:
\begin{definition}\label{def:prob frames}
Let $p\geq 1$ be an integer. 
We say that a random matrix $T\in \K^{d\times r}$ is a \emph{random g-frame of order $p$}  if there are positive constants $A_p,B_p>0$ such that 
\begin{equation}\label{eq:prob fusion frame formula}
A_p\|x\|^{2p} \leq \mathbb{E} \|T^*x\|^{2p}  \leq B_p\|x\|^{2p}, \text{ for all $x\in\K^d$.}
\end{equation} 
A random g-frame $T$ of order $p$ is called \emph{tight} if we can choose $A_p=B_p$. 
\end{definition}

Following the lines of the proof for rank one projectors considered in \cite{Ehler:2010aa} yields that any random g-frame $T$ of order $1$ satisfies $ A_1\leq \frac{1}{d}\mathbb{E} \|T(\omega)\|^2_{HS}\leq B_1$. 
Similar to finite frames, if $T$ is a random g-frame of order $p$, then the \emph{random g-frame operator}
\begin{equation}\label{eq:S}
S : \K^d \rightarrow \K^d,\quad x\mapsto  \mathbb{E}TT^* x 
\end{equation}
is positive, self-adjoint, and invertible. Thus, we obtain the reconstruction formula
\begin{equation}\label{eq:reconstr}
x = \mathbb{E} T(S^{-1} T)^*x.
\end{equation}
Moreover, $T$ is a tight random g-frame of order $1$ if and only if $S=A I_d$, where $A=\frac{1}{d}\mathbb{E} \|T\|^2_{HS}$. 

Note that the case $r=1$ of the following result is already explicitly contained in \cite{Vershynin:2010aa}:
\begin{theorem}\label{theorem:sampling}
Let $\{T_j\}_{j=1}^n$ be independent copies of a tight random g-frame $T\in\K^{d\times r}$ of order $1$ with $\|T\|^2_{HS}=R$ for some positive constant $R$. For fixed $\varepsilon\in(0,1]$, there are positive constants $\gamma_\varepsilon,c_\varepsilon>0$ such that, for all $n\geq c_\varepsilon d \ln(d)$,  the g-frame operator $S_n$ of the scaled collection $\{\sqrt{\frac{d}{nR}}T_j\}_{j=1}^n$ satisfies $\|I_d - S_n  \|_\infty  <\varepsilon$ with probability at least $1-e^{-\gamma_\varepsilon\frac{n}{d}}$ .
\end{theorem}
\begin{proof}
Let $\lambda_{\min}(S_n)$ and $\lambda_{\max}(S_n)$ denote the smallest and largest eigenvalue of $S_n$, respectively. The matrix Chernoff bounds as stated in \cite{Tropp:2011fk} yield, for all $0\leq\varepsilon\leq 1$,
\begin{align*}
\mathbb{P}\big( \lambda_{\min}(S_n)\leq 1-\varepsilon\big) & \leq d
(\frac{e^{-\varepsilon}}{(1-\varepsilon)^{1-\varepsilon}})^{n/d} \\
\mathbb{P}\big( \lambda_{\min}(S_n)\geq 1+\varepsilon\big) & \leq d(\frac{e^\varepsilon}{(1+\varepsilon)^{1+\varepsilon}})^{n/d}.
\end{align*}
Some calculus yields
\begin{equation*}
(1+\varepsilon)^{1+\varepsilon}(1-\varepsilon)^{\varepsilon-1} \leq e^{2\varepsilon},\quad\forall \varepsilon\in [0,1],
\end{equation*}
so that we derive
%
\begin{equation*}
\mathbb{P}\big(\|I_d - \sum_{j=1}^n \frac{d}{nR} T^*_jT_j  \|_\infty  \geq \varepsilon\big)\leq 2d(\frac{e^\varepsilon}{(1+\varepsilon)^{1+\varepsilon}})^{n/d}.
\end{equation*}
We can further compute
\begin{align*}
2d(\frac{e^\varepsilon}{(1+\varepsilon)^{1+\varepsilon}})^{n/d} & = 2d e^{-\frac{n}{d}( (1+\varepsilon)\ln(1+\varepsilon)-\varepsilon )}\\
& = e^{-\frac{n}{d}( (1+\varepsilon)\ln(1+\varepsilon)-\varepsilon -\frac{d}{n}\ln(2)\ln(d))}.
\end{align*}
Since $(1+\varepsilon)\ln(1+\varepsilon)-\varepsilon >0$, for all $\varepsilon\in(0,1]$, we can find a suitable constant $\gamma_\varepsilon>0$ if $n$ is sufficiently large. 
\end{proof}
\begin{remark}
The constants $\gamma_\varepsilon$ and $c_\varepsilon$ in Theorem \eqref{theorem:sampling} can be explicitly computed. By using $a_\varepsilon:=(1+\varepsilon)\ln(1+\varepsilon)-\varepsilon >0$, we can choose $c_\varepsilon>\frac{\ln(2)}{a_\varepsilon}$ and $\gamma_\varepsilon=a_\varepsilon-\frac{\ln(2)}{c_\varepsilon}$. 
\end{remark}

Next, we discuss a few examples.
\begin{example}[Gaussian matrices]\label{ex:Gaussian}
Let $1\leq k<d$ and consider the $d\times r$ random matrix $T$ whose entries are i.i.d.~Gaussian. Its joint element density is 
\begin{equation}\label{eq:density gaussian}
M\mapsto \frac{1}{(2\pi)^{kd/2}}\exp(-\frac{1}{2}\|M\|^2_{HS}).
\end{equation}
The resulting self-adjoint matrix $TT^*\in\R^{d\times d}$ is a singular Wishart-matrix, cf.~\cite{Uhlig:1994fk}. According to \eqref{eq:density gaussian} the distribution of $T$ is invariant under orthogonal transformations, so that $T$ is a tight random g-frame of order $p$, for all integers $p$. By using the moments of the chi-squared distribution, we see that the bounds satisfy $A_p=B_p=k(k+2)\cdots(k+2p-2)$. 
\end{example}

\begin{example}[fusion frames]
If the columns of a matrix $T\in \K^{d\times r}$, $r<d$,  has orthonormal columns, then we can identify $T$ with a subspaces $V\in\mathcal{G}_{r,d}(\K)$, where $\mathcal{G}_{r,d}(\K)$ denotes the Grassmann space, i.e., the collection of $r$-dimensional subspaces of $\K^d$. The Haar measure on $\mathcal{G}_{r,d}(\K)$ then induces a random g-frame of order $p$ for all integers $p$. 
\end{example}

\begin{example}[Gabor]
Time-frequency structured matrices were considered in \cite{Pfander:fk} in relation to compressed sensing, in which some window vector is modulated and shifted. We use cyclic shifts, which can be performed by applying a matrix $C$ having ones in the lower secondary diagonal, another one in the upper right corner, and zeros anywhere else. The modulation operator on $\C^d$ is given by
\begin{equation*}
M =\diag(1,e^{2\pi i/d},\ldots,e^{2\pi i (d-1)/d}).
\end{equation*}
For any nonzero $g\in\C^d$, the full Gabor system $\{M^\ell C^k g: \ell,k=0,\ldots,d-1\}$ has cardinality $d^2$ and forms a tight frame for $\C^d$, cf.~\cite{Lawrence:2005fk}. We shall use the $d^2\times d$ matrix $T$, whose rows are formed by the tight frame vectors. A short computation yields that, if $g$ is chosen at random as the Rademacher sequence, then $T$ is a tight random g-frame of order $1$. Moreover, each $TT^*$ is an orthogonal projector, so that $TT^*$ corresponds to a tight random fusion frame. The same holds when $g$ is the Steinhaus sequence, i.e., each entry is uniformly distributed on the complex unit circle.
\end{example}

Next, we have an example that indeed allows for fast matrix vector multiplication:
\begin{example}[Circulant matrices]\label{ex:circulant}
Given a vector $x=(x_1,\ldots,x_d)^\top\in\R^d$, the corresponding circulant matrix is
\begin{equation*}
\tilde{T}=\begin{pmatrix}
x_1   	& x_d 	&\hdots 	& x_2\\
x_2      	& \ddots	& \ddots 	&\vdots\\
\vdots 	& \ddots 	& \ddots &x_d\\
x_d 		& \hdots& x_2  & x_1
\end{pmatrix}.
\end{equation*}
Each column of $\tilde{T}$ is a cyclic shift of the previous one. The left $d\times k$ block $T$ of the matrix $\tilde{T}$ was used as a compressed sensing measurement matrix in \cite{Rauhut:2012fk}. If the entries of $x$ are i.i.d.~with zero mean and non-vanishing second moments, then $T$ is a tight random g-frame of order $1$ with $A_1=k\mathbb{E}(x_i^2)$. For instance, if $x$ is the Rademacher sequence, i.e., entries are independent and equal to $\pm 1$ with probability $1/2$, then $T$ is tight of order $1$ but not of order $2$ in general. It is well-known that the discrete Fourier matrices diagonalize circulant matrices, so that fast matrix vector multiplications are available. In fact, the terms ``filter bank'' and ``filterning'' are usually associated with the application of convolution operators, so that each channel corresponds to a circulant matrix with potentially some subsampling involved. 
\end{example}

\begin{remark}
Samples of all of the above examples satisfy the conditions \eqref{it:1}-\eqref{it:4} with high probability for sufficiently large sample size. The circulant matrices represent convolution operators and hence correspond to a proper filter bank scheme. They enable fast matrix vector multiplications, hence, the circulant samples in Example \ref{ex:circulant} are indeed suitable for our construction in Section \ref{sec:I} that preserves this fast algorithmic scheme using the filter bank shown in Fig.~\ref{fig:filtering}. Each channel corresponds to filtering, but we require one additional linear operator for pre- and postmultiplication. 
\end{remark}

It must be mentioned that filter banks usually involve some subsampling. Let the matrix $T\in\K^{\tilde{r}\times r}$, $\tilde{r}\geq r$ be a random matrix with a single one in each column, whose position is chosen independently at random in a uniform fashion. Then each matrix of the sample $\{T_j\}_{j=1}^n$ corresponds to a sampling operator, so that we derive $n$ samplings of length $r$. Indeed, $T$ is a tight random g-frame, but it may not satisfy all other conditions in Theorem \ref{theorem:Paul}. Nonetheless, subsampling operators in a filter bank are used in combination with more sophisticated filters, say $\{R_j\}_{j=1}^n\subset\K^{d\times \tilde{r}}$, so that it is possible that the conditions are satisfied by $\{R_j T_j\}_{j=1}^n\subset\K^{d\times r}$.

\section{Closeness to the original $g$-frame}\label{sec:III}
To relate the algorithm of the previous section to the Paulsen problem, we would need estimates on the distance between the original and the resulting $g$-frame. Especially, if the original unit norm $g$-frame is close to being tight, then we aim to verify that the computed unit norm tight $g$-frame is nearby. We do not derive any estimates for fixed $n$ but shall provide some framework for random samples that supports such intuition. 


 \begin{theorem}\label{th:consistency}
 Let $T$ be a random matrix continuously distributed on the set of matrices in $\K^{d\times r}$ and $\{T_j\}_{j=1}^n$ an associated i.i.d.~sample with $\Gamma^{(n)}$ being the corresponding limit of the iterative algorithm \eqref{eq:def Gamma}. Then $\Gamma^{(n)}$ converges almost surely towards some positive definite $\Gamma$, so that  $\Sigma:=\Gamma^{-1}$ satisfies 
 \begin{equation}\label{eq:final eq}
\Sigma=d\mathbb{E} \frac{TT^*}{\trace(T^*\Sigma^{-1} T)}.
\end{equation}
 \end{theorem}
As in \cite{Tyler:1987fk}, we observe that results of the previous section applied to a continuously distributed random matrix $T$ yield that \eqref{eq:final eq} has a solution $\Sigma$ among the symmetric positive definite matrices and is unique up to multiplication by a positive constant. 

For elliptical distributions, \eqref{eq:final eq} has a very special meaning. Here, we call a probability distribution on $\K^{d\times r}$ \emph{elliptical} if it has a density $f$ with respect to the standard volume element $dT$ on $\K^{d\times r}$ and 
\begin{equation*}
f(T) = |\det(\Sigma)|^{-1/2} g(\trace((T-\bar{T})^*\Sigma^{-1}(T-\bar{T}))), 
\end{equation*}
where $\Sigma\in\K^{d\times d}$ is hermitian positive definite, $\bar{T}\in\K^{d\times r}$, and $g$ is some nonnegative function not dependent on $\bar{T}$ and $\Sigma$ with $\int_{\K^{d\times r}} g(\trace(T^*T)) dT = 1$. For instance, the Gaussian random matrix in Example \ref{ex:Gaussian} is elliptically distributed. A direct computation yields that the matrix $\Sigma$ of an elliptically distributed random matrix $T$ with $\bar{T}=0$ satisfies \eqref{eq:final eq}. For simplicity, we shall restrict us to the case $\bar{T}=0$ and point out that general $\bar{T}$ can be handled in a similar fashion, see \cite{Tyler:1987fk} for $r=1$. 

Theorem \ref{th:consistency} directly implies the following:
\begin{corollary}
If $T$ is elliptically distributed with $\bar{T}=0$ and $\Sigma$ is a multiple of the identity, then, for any sample $\{T_j\}_{j=1}^n$, the associated matrices $\Gamma^{(n)}$, for $n\rightarrow \infty$, converge towards $\frac{1}{d}I$.   
\end{corollary}

 To verify Theorem \ref{th:consistency}, we follow the ideas in \cite{Tyler:1987fk}, where $r=1$ was considered, so we need some notation and two lemmas. Let us define
 \begin{equation*}
 M_n(\Gamma)=\frac{d}{n} \sum_{j=1}^n \frac{\Gamma^{1/2}T_j T^*_j\Gamma^{1/2}}{\trace(T^*_j \Gamma T_j)},\quad\quad M(\Gamma)= d\mathbb{E}\frac{\Gamma^{1/2}TT^*\Gamma^{1/2}}{\trace(T^*\Gamma T)},
 \end{equation*}
 and we denote $h_n(\Gamma):=\trace(M_n(\Gamma)^2)$ and $h(\Gamma):=\trace(M(\Gamma)^2)$.
 \begin{lemma}\label{lemma:1a}
 Let $\mathcal{C}\subset \K^{d\times d}$ be a compact set of positive definite matrices with $R\in \mathcal{C}$ implying $\trace(R^{-1})=1$ and $\trace(R)\leq K$ with some fixed $K>1$. Then 
 \begin{equation}\label{eq:strong}
 \sup_{R\in \mathcal{C}} |h_n(R)-h(R)|\rightarrow 0
 \end{equation}
holds almost surely.
 \end{lemma}
 \begin{lemma}\label{lemma:1b}
 The hermitian positive definite matrix $R\in\K^{d\times d}$ is a critical value of $h_n$ if and only if $M_n(R)=I$.
 \end{lemma}

 \begin{proof}[Proof of Theorem \ref{th:consistency}]
Simple arithmetics yield that $\trace(M_n(R))=\trace(M(R))=d$ implies $d\leq h_n(R),h(R)\leq d^2$, for all positive definite matrices $R\in\K^{d\times d}$. Furthermore, $h_n(R)=d$ or $h(R)=d$ if and only if $M_n(R)=I$ or $M(R)=I$, respectively. 
 
 As mentioned above, \eqref{eq:final eq} has a solution $\Sigma$ among the symmetric positive definite matrices and is unique up to multiplication by a positive constant, see also \cite{Tyler:1987fk}. Without loss of generality, we can assume that $\Sigma=dI$, which implies $M(\Sigma^{-1})=M(\frac{1}{d}I)=I$. Choose $C$ as in Lemma \ref{lemma:1b} with $\frac{1}{d}I$ being contained in its interior. For all $R\in \mathcal{C}$ with $R\neq \frac{1}{d}I$, we must have $d=h(\frac{1}{d}I)<h(R)$. Since $h$ is continuous, Lemma \ref{lemma:1a} yields that, for any $R$ on the boundary of $\mathcal{C}$,  we have $h_n(\frac{1}{d}I)<h_n(R)$ with probability one if $n$ is sufficiently large.  
 
 Note that $M_n(\Gamma^{(n)})=I$, so that Lemma \ref{lemma:M} and Lemma \ref{lemma:1b} imply that $\Gamma^{(n)}$ is eventually contained in $\mathcal{C}$ for sufficiently large $n$. Since $\mathcal{C}$ can be chosen arbitrarily small, it follows that $\Gamma^{(n)}\rightarrow \frac{1}{d}I=\Sigma^{-1}$ almost surely, which concludes the proof. 
\end{proof}
It remains to prove the two Lemmas \ref{lemma:1a} and \ref{lemma:1b}.
\begin{proof}[Proof of Lemma \ref{lemma:1a}]
We follow \cite[Proof of Statement (3.2)]{Tyler:1987fk}: 
For $0\neq T\in\K^{d\times r}$, we define 
\begin{equation*}
G(T,R):=\frac{R^{1/2} TT^*R^{1/2}}{\trace(T^*RT)}.
\end{equation*}
As already mentioned in \cite{Tyler:1987fk} for $r=1$, $G$ is equicontinuous on $\mathcal{C}$ meaning that, for $\varepsilon>0$, there is $\delta_\varepsilon>0$ not dependent on $T\neq 0$ nor on $R_1,R_2\in\mathcal{C}$, such that $\|R_1-R_2\|<\delta_\varepsilon$ implies $\|G(T,R_1)-G(T,R_2)\|_{HS}<\varepsilon$. Next, the same covering argument for $\mathcal{C}$ as in \cite{Tyler:1987fk} used with the equicontinuity and the strong law of large numbers implies \eqref{eq:strong}. We omit the details.   
 \end{proof}

\begin{proof}[Proof of Lemma \ref{lemma:1b}]
We can simply follow the lines of \cite[Proof of Statement (3.3)]{Tyler:1987fk}, where $r=1$ is discussed. A first order expansion of $h_n$ with the frame property and Kantorovich's inequality yields Lemma \ref{lemma:1b}. No new ideas are involved when dealing with $r>1$, so we refer to \cite{Tyler:1987fk} for the details. 
\end{proof}

\section{Modifications and numerical experiments}
We have claimed that the additional property of equal norms in filter banks with perfect reconstruction can be beneficial  compared to the original filter bank. The present section is dedicated to support this statement by some modifications and numerical denoising experiments using wavelet filters applied to the Lena and Peppers images, cf.~Fig.~\ref{fig:org images}. 

It should be mentioned that PDE based methods seem to provide better denoising results in images than pure wavelet based schemes. Therefore, our use in wavelet based denoising is rather an example than the main application, for which we developed the theoretical scheme, and we only provide the root mean squared error (RMSE) but suppress the denoised images themselves. 
\begin{figure}
\subfigure[Lena]{\includegraphics[]{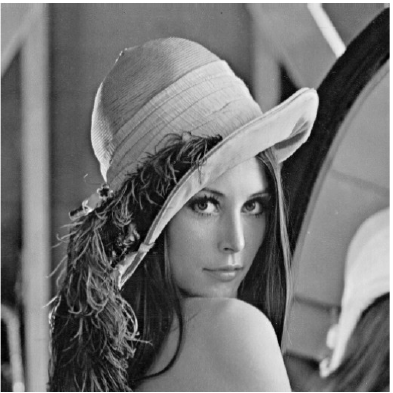}}\;\;
\subfigure[Peppers]{\includegraphics[]{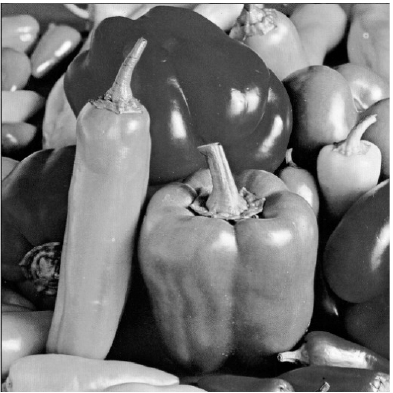}}
\caption{The two test signals, for which we perform denoising experiments.}\label{fig:org images}
\end{figure}

Given a collection of wavelet frame filters $\{T_j\}_{j=1}^n$, we observe that $\Gamma$ is the inverse of the frame operator of the reweighted filters $
\{c_j T_j\}_{j=1}^n$, where $c_j = \sqrt{\frac{d}{n\trace(T_j^*\Gamma T_j)}}$. Since each $T_j$ is supposed to be a filter, there are then fast algorithms to apply $\Gamma$ within the filter bank scheme, cf.~\cite{Wiesmeyr:2013fk}. Since it may not be clear how the premultiplication of $\Gamma^{1/2}$ changes the signal before the actual filters $\{T_j\}_{j=1}^n$ come into play in Fig.~\ref{fig:filtering}, we better modify the scheme and work with the filter bank designed in Fig.~\ref{fig:filteringNEW}. Note that the knowledge that $\Gamma$ is the inverse of a frame operator is essential to build this new filter bank.

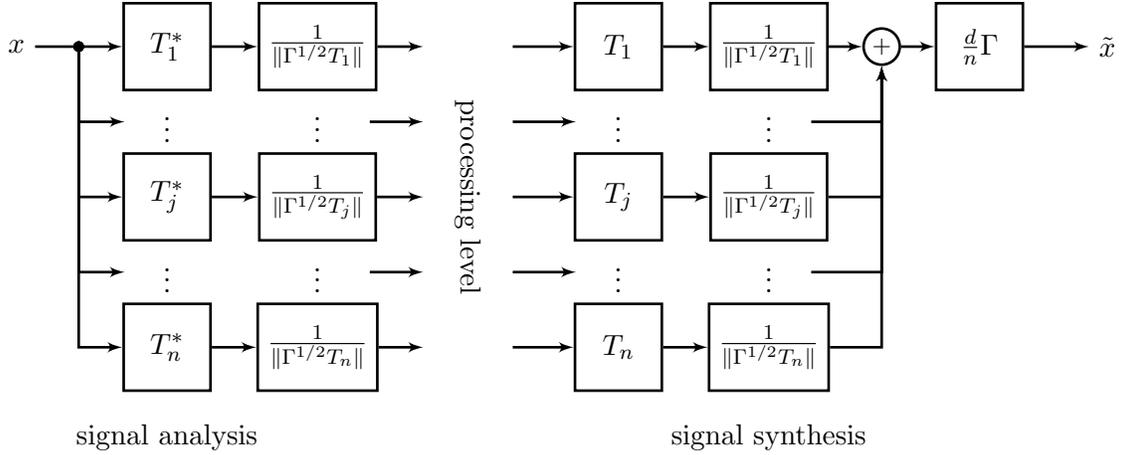
\begin{figure}
\begin{tikzpicture}[auto,>=latex']
    \tikzstyle{block} = [draw, shape=rectangle, minimum height=3em, minimum width=3em, node distance=2cm, line width=1pt]
    \tikzstyle{block2} = [draw, shape=rectangle, minimum height=3em, minimum width=3em, node distance=1.3cm, line width=1pt]
    \tikzstyle{phant} = [shape=rectangle, minimum height=3em, minimum width=3em, node distance=2cm, line width=1pt]
  \tikzstyle{phant2} = [shape=rectangle, minimum height=0em, minimum width=3em, node distance=1.2cm, line width=1pt]
    \tikzstyle{dots} = [shape=circle,minimum height=1em, minimum width=2em, node distance=1cm]
    \tikzstyle{sum} = [draw, shape=circle, node distance=1.5cm, line width=1pt, minimum width=1.25em]
    \tikzstyle{branch}=[fill,shape=circle,minimum size=4pt,inner sep=0pt]
    \node at (-1.5,0) (input) {$x$};
    \node at (-2.5,0) [block, right of=input] (T1) {$T_1^*$};
    \node [dots, below of=T1] (D1) {$\;\;\vdots\;\;$};
     \node [block, right of=T1] (H1) {$\frac{1}{\|\Gamma^{1/2}T_1\|}$};
       \node [dots, below of=H1] (D2) {$\quad\vdots\quad$};
         \node [block, below of=H1] (Hj) {$\frac{1}{\|\Gamma^{1/2}T_j\|}$};
         \node [dots, below of=Hj] (Hn) {$\quad\vdots\quad$};
           \node [block, below of=Hj] (Hnd) {$\frac{1}{\|\Gamma^{1/2}T_n\|}$};
             \node [phant, right of=H1] (P1) {};
                         \node [phant, right of=D2] (P2) {};
                                                  \node [phant, right of=Hn] (P3) {};
                                                  			    \node [phant, right of=Hn] (Pn) {};
			    \node [phant, right of=Hnd] (Pnd) {};
      \node [block, right of=P1] (TT1) {$T_1$};
        \node [block, right of=TT1] (HH1) {$\frac{1}{\|\Gamma^{1/2}T_1\|}$};
         
                    \node [dots, below of=HH1] (HH2) {$\;\;\vdots\;\;$};
    \node [sum, right of=HH1] (sum) {};
    \node [block2, right of=sum] (G2) {$\frac{d}{n}\Gamma$};

          \node [phant, right of=Hj] (Pj) {\begin{turn}{-90} processing  level\end{turn}};
    
    \node at (sum) (plus) {{\footnotesize$+$}};
     
    \node at (13,0) (output) {$\tilde{x}$};
    \path (input) -- coordinate (med) (T1);
    \node [block, below of=T1] (Tj) {$T_j^*$};
      \node [dots, below of=Tj] (Tn) {$\;\;\vdots\;\;$};
       \node [block, below of=Tj] (Tnd) {$T_n^*$};    
           \node [phant2, below of=Tnd] (END1) {signal analysis};     
                  \node [dots, below of=TT1] (TT2) {$\;\;\vdots\;\;$};    
       \node [block, below of=TT1] (TTj) {$T_j$};    
          \node [dots, below of=TTj] (TTn) {$\;\;\vdots\;\;$};    
       \node [block, below of=TTj] (TTnd) {$T_n$};    
        \node [block, right of=TTj] (HHj) {$\frac{1}{\|\Gamma^{1/2}T_j\|}$};
                            \node [dots, below of=HHj] (HHn) {$\;\;\vdots\;\;$};
                                    \node [block, right of=TTnd] (HHnd) {$\frac{1}{\|\Gamma^{1/2}T_n\|}$};
                                       \node [phant2, below of=HHnd] (END2) {signal synthesis};     
    \begin{scope}[line width=1pt]
         \draw[->] (input) -- (T1);
           \draw[->] (Tj) -- (Hj);
                \draw[->] (D2) -- (P2);
                     \draw[->] (H1) -- (P1);
                      \draw[->] (Hn) -- (Pn);
                      \draw[->] (Hnd) -- (Pnd);                     
     \draw[->] (Hj) -- (Pj);
    \draw[->] (Tnd) -- (Hnd);
            \draw[->] (G2) -- (output);
                   \draw[->] (P1) -- (TT1);    
                            \draw[->] (TT1) -- (HH1);    
                   
                     \draw[->] (P2) -- (TT2);
                   \draw[->] (Pj) -- (TTj);    
                                      \draw[->] (Pnd) -- (TTnd);    
                                                                            \draw[->] (Pn) -- (TTn);    
                           
           \draw[->] (med) node[branch] {} |- (Tj);
               \draw[->] (med) node[branch] {} |- (D1);
             \draw[->] (med) node[branch] {} |- (Tn);
                \draw[->] (med) node[branch] {} |- (Tnd);
                 \draw[->] (T1) -- (H1);
                 
                          \draw[->] (TTj) -- (HHj);
			 \draw[->] (TTnd) -- (HHnd);  
			 	 \draw[->] (HH1) -- (sum);    
                 \draw[->] (HHj) -| (sum);
                  \draw[->] (HH2) -| (sum);
                 \draw[->] (HHn) -| (sum);
                  \draw[->] (HHnd) -| (sum);
                          \draw[->] (sum) -- (G2);   
                 
    \end{scope}
\end{tikzpicture}
\caption{Modified analysis and synthesis scheme, in which $\{T_j\}_{j=1}^n$ get reweighted, and for those we postmultiply with its inverse frame operator. Since, $\Gamma$ is an inverse frame operator, its application can be performed using fast approximate methods, see \cite{Wiesmeyr:2013fk}. }\label{fig:filteringNEW}
\end{figure}

We shall compare our approach with three types of alternative reconstruction schemes. As shown in Fig.~\ref{fig:filtering b}, the alternatives are first taking the canonical dual as synthesis and, secondly, taking it as the analysis part in the filter bank. 
\begin{figure}
\subfigure[Original frame is used for analysis, and synthesis is computed by the canonical dual frame.]{
\includegraphics[height=.25\textwidth]{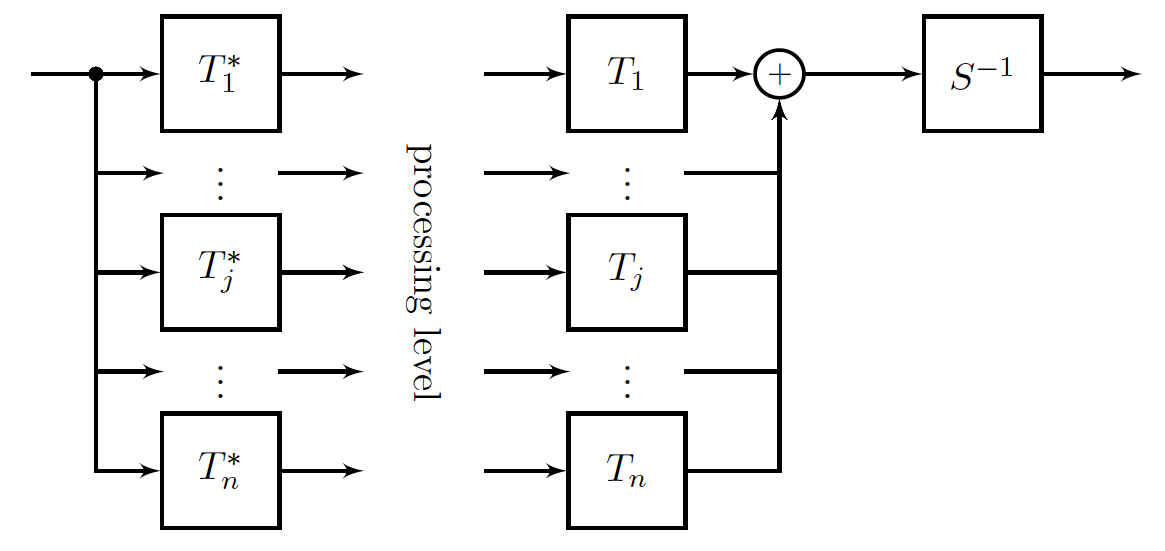}}\hfill
\subfigure[Original frame is used for analysis, and synthesis is derived through dual wavelet filters.]{
\includegraphics[height=.25\textwidth]{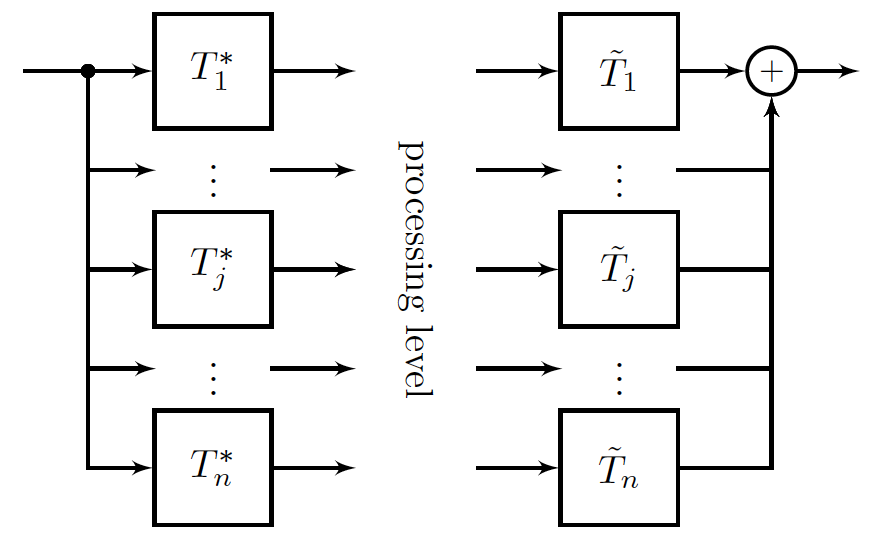}}

\subfigure[Original frame is replaced with its canonical tight frame. The filter bank yields perfect reconstruction but operators in each channel do not have equal norm.]{
\includegraphics[height=.25\textwidth]{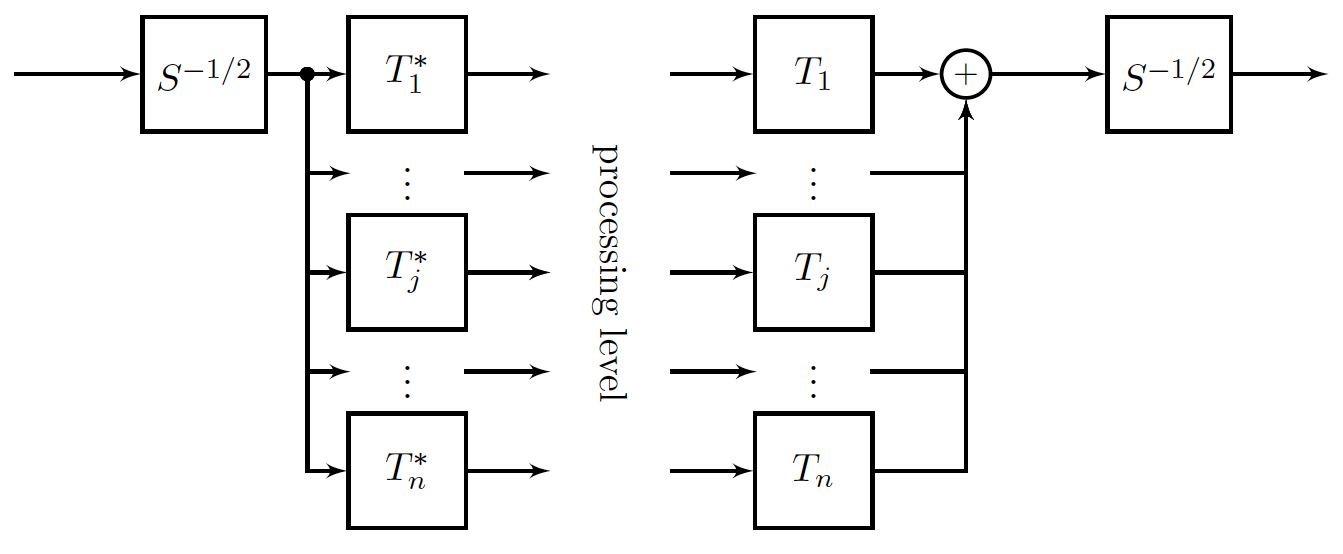}}

\caption{Given a set of linear operators $\{T_j\}_{j=1}^n$ associated to a wavelet system, there are several ways to build a filter bank with exact reconstruction. We used the canonical dual frame, a dual wavelet frame, and the canonical tight frame.}\label{fig:filtering b}
\end{figure}
For the third alternative, we use the primal frame and a dual that has some beneficial properties, namely also being induced by wavelets. Indeed, the given wavelet filters (or better convolution operators with down-sampling, see also \cite{Chebira:2011fk}) are chosen by the multiwavelet approach in \cite{Ehler:aa}, and the Lena image is corrupted by  additive Gaussian white noise of zero mean and three different variances, so that the signal to noise ratio is $10$, $25$, and $40$, respectively.

Tables \ref{table 2} and \ref{table 1} show that our filter bank design can improve on the denoising performance of wavelet filter banks. For low noise levels, our approach yields slightly higher RMSE values but with increasing noise levels the RMSE is clearly lower than for the three other methods we compare with. This behavior is consistent in Lena and Peppers. 

Since the construction of  $\Gamma$ only depends on the original wavelet frame and on the decomposition scale, it seems reasonable to ignore the computation time of $\Gamma$ when considering runtime comparisons. Nonetheless, it must be mentioned that our approach is much slower than the pairs of dual wavelets in \cite{Ehler:aa} because we need to apply twice an operator that is highly structured but not exactly a convolution. To avoid storage and runtime issues, we chopped the image into smaller pieces enabling us to perform all computations in matlab. The size of those pieces did not seem to have much effect as long as some lower threshold is avoided.

\begin{table}
\begin{tabular}{|c||c|c|c|c|}
\hline
SNR & our approach & canonical dual synthesis & canonical dual analysis & pair of dual wavelets\\
\hline
10 & 5.98 & 6.12 &  5.96  & 5.92 \\
25 & 9.94  & 10.87 & 10.86 & 10.25\\
40 & 11.98& 14.21& 13.79 & 13.11\\
\hline

\end{tabular}
\caption{Lena: root mean squared error (RMSE) for different denoising filter banks using soft-thresholding and varying signal to noise ratio (SNR). The multiwavelet frame operators $\{T_j\}_{j=1}^n$ correspond to the Laplace-$\Phi^v_2$ frame constructed in \cite{Ehler:aa}. The best threshold is computed by brute force methods. }\label{table 2}
\end{table}

\begin{table}
\begin{tabular}{|c||c|c|c|c|}
\hline
SNR & our approach & canonical dual synthesis & canonical dual analysis & pair of dual wavelets\\
\hline
10 &  6.28 & 6.17 & 6.09  & 6.01 \\
25 &  9.95  & 10.52  & 10.38 &  10.34\\
40 &  12.27 & 14.02 &  13.96 & 13.39 \\
\hline
\end{tabular}
\caption{Peppers: root mean squared error (RMSE) for different denoising filter banks using soft-thresholding and varying signal to noise ratio (SNR). The multiwavelet frame operators $\{T_j\}_{j=1}^n$ correspond to the Laplace-$\Phi^v_2$ frame constructed in \cite{Ehler:aa}. The best threshold is computed by brute force methods. }\label{table 1}
\end{table}


\section{Conclusions}
For some signal processing aspects, the most attractive filter bank schemes are those that provide perfect reconstruction, synthesis is the adjoint of the analysis scheme (so-called unitary filter banks), and filters have equal norm. Tight fusion frames, for instance, correspond to perfect reconstruction filter banks, in which each channel corresponds to an orthogonal projection, and it was verified in \cite{Kutyniok:2009aa} that robustness of tight fusion frames against distortions and erasures is maximized when the tight fusion frame has equal norm elements. Our aim was to turn a given filter bank into such more attractive schemes and preserving the essential features of the original filtering process. In terms of frames, we turned a given generalized frame into a tight g-frame with unit norm by rescaling and then applying the inverse square root of the new g-frame operator. Due to our special focus on filter banks, we started with a generalized frame consisting of convolution operators, hence, allowing for fast matrix vector multiplications. Through some iterative scheme, we constructed a generalized tight frame with unit norm, which induced a filter bank that preserved the convolution structure, hence, the fast algorithmic scheme, in each channel. Only one additional global pre- and postmultiplication with $\Gamma^{1/2}$ is necessary.  Naturally, the application of $\Gamma^{1/2}$ needs special care because it may be structured but not exactly a convolution operator. In our numerical experiments, we omitted the pre-multiplication but post-multiplied by $\Gamma$, which turned out to be the inverse of a frame operator and for which then fast computational schemes are available, cf.~\cite{Wiesmeyr:2013fk}. Other alternatives are special matrix storage management and chopping images into smaller pieces to reduce the matrix size to deal with. Nonetheless, this procedure is only applied once or twice globally and not in each channel, so that runtime does not grow with additional filters when compared to the original frame. 

We observed that the assumptions of our algorithm are satisfied by any sufficiently large sample drawn from any continuous distribution or drawn from random convolution operators. Fields of application are filter banks, in which the additional computation costs of the application of $\Gamma^{1/2}$ or $\Gamma$, respectively, can be tolerated, as for instance, when the number of channels is large or when computations are completely off-line. 

Our findings provide a tool to design new filter banks with improved properties on a theoretical level. Our numerical experiments seem to suggest improvements in wavelet based denoising but are far from being comprehensive. Further examples and wider applications in several fields go beyond the scope of the present manuscript but are needed to fully verify its usefulness from an application point of view. We hope that our theoretical findings provide the basis for its use in more elaborate signal processing methods.

\section*{Acknowledgements} 
The author has been funded by the Vienna Science and Technology Fund (WWTF) through project VRG12-009.


\begin{thebibliography}{10}

\bibitem{Bachoc:2010aa}
C.~Bachoc and M.~Ehler, \emph{Tight $p$-fusion frames}, Appl.~ Comput.~
  Harmon.~ Anal. \textbf{35} (2013), no.~1, 1--15.

\bibitem{Benedetto:2003aa}
J.~J. Benedetto and M.~Fickus, \emph{Finite normalized tight frames}, Adv.~
  Comput.~ Math. \textbf{18} (2003), no.~2-4, 357--385.

\bibitem{Bodmann:2010fk}
B.~G. Bodmann and P.~G. Casazza, \emph{The road to equal-norm parseval frames},
  J.~ Funct.~ Anal. \textbf{258} (2010), 397--420.

\bibitem{Cahill:2011ys}
J.~Cahill and P.~G. Casazza, \emph{The {P}aulsen problem in operator theory},
  submitted to Operators and Matrices (2011).

\bibitem{Casazzaa:2010fk}
P.~G. Casazza, M.~Fickus, and D.~G. Mixon, \emph{Auto-tuning unit norm frames},
  Appl.~ Comput.~ Harmon.~ Anal. \textbf{32} (2012), no.~1, 1--15.

\bibitem{G.-Kutyniok:2007fk}
P.~G. Casazza and G.~Kutyniok, \emph{A generalization of {G}ram-{S}chmidt
  orthogonalization generating all {P}arseval frames}, Adv.~ Comput.~ Math.
  (2007), 65--78.

\bibitem{Casazza:2008aa}
P.~G. Casazza, G.~Kutyniok, and S.~Li, \emph{Fusion frames and distributed
  processing}, Appl.~ Comput.~ Harmon.~ Anal. \textbf{25} (2008), no.~1,
  114--132.

\bibitem{Chebira:2011fk}
A.~Chebira, M.~Fickus, and D.~G. Mixon, \emph{Filter bank fusion frames}, IEEE
  Trans.~ Signal Process. \textbf{59} (2011), no.~3, 953--963.

\bibitem{Christensen:2003aa}
O.~Christensen, \emph{{A}n {I}ntroduction to {F}rames and {R}iesz {B}ases},
  Birkh{\"{a}}user, Boston, 2003.

\bibitem{Ehler:2010aa}
M.~Ehler, \emph{Random tight frames}, J.~ Fourier Anal.~ Appl. \textbf{18}
  (2012), no.~1, 1--20.

\bibitem{Ehler:2010ac}
M.~Ehler and J.~Galanis, \emph{Frame theory in directional statistics}, Stat.~
  Probabil.~ Lett. \textbf{81} (2011), no.~8, 1046--1051.

\bibitem{Ehler:aa}
M.~Ehler and K.~Koch, \emph{The construction of multiwavelet bi-frames and
  applications to variational image denoising}, Int.~ J.~ Wavelets,
  Multiresolut.~ Inf.~ Process. \textbf{8} (2010), no.~3, 431--455.

\bibitem{Okoudjou:2010aa}
M.~Ehler and K.~Okoudjou, \emph{Minimization of the probabilistic $p$-frame
  potential}, J.~ Stat.~ Plann.~ Inference \textbf{142} (2012), no.~3,
  645--659.

\bibitem{Frank:2002uq}
M.~Frank, V.~I. Paulsen, and T.~R. Tiballi, \emph{Symmetric approximation of
  frames}, Trans.~ Amer.~ Math.~ Soc. \textbf{354} (2002), 777--793.

\bibitem{Kent:1988kx}
J.~T. Kent and D.~E. Tyler, \emph{Maximum likelihood estimation for the wrapped
  {C}auchy distribution}, Journal of Applied Statistics \textbf{15} (1988),
  no.~2, 247--254.

\bibitem{Kutyniok:2013fk}
G.~Kutyniok, K.~A. Okoudjou, F.~Philipp, and E.~K. Tuley, \emph{Scalable
  frames}, Lin.~Alg.~Appl. \textbf{438} (2013), 2225--2238.

\bibitem{Kutyniok:2009aa}
G.~Kutyniok, A.~Pezeshki, R.~Calderbank, and T.~Liu, \emph{Robust dimension
  reduction, fusion frames, and {G}rassmannian packings}, Appl.~Comput.~
  Harmon.~ Anal. \textbf{26} (2009), no.~1, 64--76.

\bibitem{Lawrence:2005fk}
Jim Lawrence, G\"otz Pfander, and David Walnut, \emph{Linear independence of
  {G}abor systems in finite dimensional vector spaces}, J.~ Fourier Anal.~
  Appl. \textbf{11} (2005), no.~6, 715--726.

\bibitem{Pfander:fk}
G.~Pfander, H.~Rauhut, and J.~Tropp, \emph{The restricted isometry property for
  time-frequency structured random matrices}, Prob.~ Theory Rel.~ Fields
  (2012), 1--31.

\bibitem{Rauhut:2012fk}
H.~Rauhut, J.~Romberg, and J.~Tropp, \emph{Restricted isometries for partial
  random circulant matrices}, Appl.~ Comput.~ Harmon.~ Anal. \textbf{32}
  (2012), no.~2, 242--254.

\bibitem{Sun:2006fk}
W.~Sun, \emph{{$G$}-frames and {$G$}-{R}iesz bases}, J. Math. Anal. Appl.
  \textbf{322} (2006), 437--452.

\bibitem{Tropp:2011fk}
J.~A. Tropp, \emph{User-friendly tail bounds for sums of random matrices},
  Journal Foundations of Computational Mathematics \textbf{12} (2012), no.~4,
  389--434.

\bibitem{Tyler:1987fk}
D.~E. Tyler, \emph{A distribution-free {$M$}-estimate of multivariate scatter},
  Annals of Statistics \textbf{15} (1987), no.~1, 234--251.

\bibitem{Tyler:1987uq}
\bysame, \emph{Statistical analysis for the angular central {G}aussian
  distribution}, Biometrika \textbf{74} (1987), no.~3, 579--590.

\bibitem{Uhlig:1994fk}
H.~Uhlig, \emph{On singular {W}ishart and singular multivariate {B}eta
  distributions}, The Annals of Statistics \textbf{22} (1994), no.~1, 395--405.

\bibitem{Vershynin:2010aa}
R.~Vershynin, \emph{How close is the sample covariance matrix to the actual
  covariance matrix?}, Journal of Theoretical Probability \textbf{25} (2012),
  655--686.

\bibitem{Wiesmeyr:2013fk}
C.~Wiesmeyr, N.~Holighaus, and P.~Sondergaard, \emph{Efficient algorithms for
  discrete gabor transforms on a nonseparable lattice}, IEEE Trans.~ Signal
  Process. \textbf{61} (2013), no.~20, 5131--5142.

\end{thebibliography}

 \providecommand{\bysame}{\leavevmode\hbox to3em{\hrulefill}\thinspace}
\providecommand{\MR}{\relax\ifhmode\unskip\space\fi MR }
\providecommand{\MRhref}[2]{%
  \href{http://www.ams.org/mathscinet-getitem?mr=#1}{#2}
}
\providecommand{\href}[2]{#2}

\end{document}